\documentclass[11pt,reqno]{amsart}

\newcommand\be{\begin{equation}}
\newcommand\ee{\end{equation}}

\usepackage{times}
\usepackage[T1]{fontenc}
\usepackage{mathrsfs}
\usepackage{latexsym}
\usepackage[dvips]{graphics}
\usepackage{epsfig}
\usepackage{amsmath,amsfonts,amsthm,amssymb,amscd}
\input amssym.def
\input amssym.tex
\usepackage{color}
\usepackage{mathtools}
\usepackage{hyperref}
\usepackage{url}
\newcommand{\bburl}[1]{\textcolor{blue}{\url{#1}}}

\newtheorem{theorem}{Theorem}[section]
\newtheorem{lemma}[theorem]{Lemma}

\newtheorem{corollary}[theorem]{Corollary}

\newtheorem*{lemma*}{Lemma}
\newtheorem*{claim*}{Claim}
\newtheorem*{proposition*}{Proposition}
\newtheorem*{fact*}{Fact}
\newtheorem*{corollary*}{Corollary}
\newtheorem*{hint*}{Hint}

\theoremstyle{definition}
\newtheorem{definition}[theorem]{Definition}

\newtheorem*{theorem*}{Theorem}
\newtheorem*{definition*}{Definition}
\newtheorem*{remark*}{Remark}
\newtheorem*{notation*}{Notation}
\newtheorem*{example*}{Example}
\newtheorem*{examples*}{Examples}
\newtheorem*{question*}{Question}
\newtheorem*{problem*}{Problem}
\newtheorem*{solution*}{Solution}
\newtheorem*{intuition*}{Intuition}
\newtheorem*{idea*}{Idea}

\newcommand{\ignore}[1]{}

\newcommand\tsup{\textsuperscript}

\newcommand{\m}[1]{$#1$\,}

\newcommand{\lp}{\left(}
\newcommand{\rp}{\right)}
\newcommand{\pa}[1]{\lp#1\rp}
\newcommand{\lb}{\left[}
\newcommand{\rb}{\right]}
\newcommand{\ba}[1]{\lb#1\rb}

\newcommand{\lf}{\lfloor}
\newcommand{\rf}{\rfloor}

\newcommand{\floor}[1]{\lf {#1} \rf}

\providecommand{\abs}[1]{\left\vert#1\right\vert}

\providecommand{\abs}[1]{\left\vert#1\right\vert}

\newcommand\ii{\item}

\newcommand{\defeq}{\vcentcolon=}

\newcommand{\id}[1]{^{(#1)}}

\newcommand\half {\frac{1}{2}}

\renewcommand{\Pr}{\mathop{\bf Pr\/}}
\newcommand{\E}{\mathop{\bf E\/}}

\newcommand{\Var}{\mathop{\bf Var\/}}

\newcommand\ZZ{\mathbb{Z}}
\newcommand\RR{\mathbb{R}}

\newcommand\NN{\mathbb{N}}

\newcommand\init{\text{init}}

\newcommand\bea{\begin{eqnarray}}
\newcommand\eea{\end{eqnarray}}
\newcommand\ben{\begin{enumerate}}
\newcommand\een{\end{enumerate}}

\numberwithin{equation}{section}

\begin{document}

\title{Central Limit Theorems for Gaps of Generalized Zeckendorf Decompositions}

\author{Ray Li}
\address{\tiny{Department of Mathematics, Carnegie Mellon University, Pittsburgh, PA 15213}}
\email{ryli@andrew.cmu.edu}

\author{Steven J. Miller}
\address{\tiny{Department of Mathematics and Statistics, Williams College, Williamstown, MA 01267}}
\email{sjm1@williams.edu, Steven.Miller.MC.96@aya.yale.edu}


\subjclass[2010]{60B10, 11B39, 11B05  (primary) 65Q30 (secondary)}

\keywords{Zeckendorf decompositions, Central Limit Theorem, recurrence relations}

\date{\today}

\thanks{The second named author was partially supported by NSF grants DMS1265673 and DMS1561945. The authors thank their colleagues from Math 21-499 at Carnegie Mellon University and CANT 2016 for many helpful conversations.}

\begin{abstract} Zeckendorf proved that every integer can be written uniquely as a sum of non-adjacent Fibonacci numbers $\{1,2,3,5,\dots\}$. This has been extended to many other recurrence relations $\{G_n\}$ (with their own notion of a legal decomposition) and to proving that the distribution of the number of summands of an $M \in [G_n, G_{n+1})$ converges to a Gaussian as $n\to\infty$. We prove that for any non-negative integer $g$ the average number of gaps of size $g$ in many generalized Zeckendorf decompositions is $C_\mu n+d_\mu+o(1)$ for constants $C_\mu > 0$ and $d_\mu$ depending on $g$ and the recurrence, the variance of the number of gaps of size $g$ is similarly $C_\sigma n + d_\sigma + o(1)$ with $C_\sigma > 0$, and the number of gaps of size $g$ of an $M\in[G_n,G_{n+1})$ converges to a Gaussian as $n\to\infty$. The proof is by analysis of an associated two-dimensional recurrence; we prove a general result on when such behavior converges to a Gaussian, and additionally re-derive other results in the literature.
\end{abstract}

\maketitle

\tableofcontents

\section{Introduction}



\subsection{Previous Results}

Zeckendorf \cite{Ze} proved that if the Fibonacci  numbers are defined by $F_1 = 1, F_2 = 2$ and $F_{n+1} = F_n + F_{n-1}$, then every integer can be written as a sum of non-adjacent terms. The standard proof is by the greedy algorithm, though combinatorial approaches exist (see \cite{KKMW}). More generally, one can consider other sequences of numbers and rules for a legal decomposition, and ask when a unique decomposition exists, and if it does how the summands are distributed.

There has been much work on these decomposition problems. In this paper we concentrate on the class of positive linear recurrences (see \cite{Al, DDKMV} for signed decompositions, \cite{DDKMMV} for $f$-decomposition, and \cite{CFHMN1, CFHMN2, CFHMNPX} for some recurrences where the leading term vanishes, which can lead to different limiting behavior).

  \begin{definition}
    A \textit{positive linear recurrence sequence} (\emph{PLRS})
    is a sequence $\{G_n\}$ satisfying
    \begin{equation}
      G_n\ = \ c_1G_{n-1}+\cdots+c_LG_{n-L}
    \end{equation}
    with non-negative integer coefficients $c_i$ with $c_1,c_L,L\ge 1$
    and initial conditions $G_1=1$ and
    $G_{n}\ = \ c_1G_{n-1}+c_2G_{n-2}+\cdots+c_{n-1}G_1+1$
    for $1\le n\le L$.
  \end{definition}

Informally a legal decomposition is one where we cannot use the recurrence relation to replace a linear combination of summands with another summand, and the coefficient of each summand is appropriately bounded. We first describe four results on these sequences (see \cite{DG1998, Ha, Ho, Ke, Len, LT, MW1, MW2, PT1989, Steiner2002, Ste2}, especially \cite{MW1} for proofs), and then discuss our new work.

  \begin{theorem}[Generalized Zeckendorf Theorem]
    \label{thm:general-zeck}
    Let $\{G_n\}$ be a positive linear recurrence sequence.
    For each integer $M>0$, there exists a
    unique \textbf{legal} decomposition
    \begin{equation}
      M\ = \ \sum_{i=1}^N a_iG_{N+1-i}
    \end{equation}
    with $a_1>0$ and the other $a_i\ge 0$,
    and one of the following two conditions,
    which define a legal decomposition, holds.
    \begin{enumerate}
      \ii We have $N<L$ and $a_i=c_i$ for $1\le i\le N$.
      \ii There exists an $s\in\{1,\dots,L\}$ such that
      $a_1=c_1,a_2=c_2,\dots,a_{s-1}=c_{s-1}$
      and $a_s<c_s$,
      $a_{s+1},\dots,a_{s+\ell}=0$ for some $\ell\ge0$,
      and $\{b_i\}_{i=1}^{N-s-\ell}$ (with $b_i=a_{s+\ell+i}$)
      is either legal or empty.
    \end{enumerate}
  \end{theorem}

The  next result concerns the average number of summands in decompositions, generalizing Lekkerkerker's \cite{Lek} work on this problem for the Fibonacci numbers. Given  $\{G_n\}$ a PLRS, we have the legal decomposition
    \begin{equation}
      M \ = \  \sum_{i=1}^N a_iG_{N+1-i}  \ = \  G_{i_1} + G_{i_2} + \cdots + G_{i_k}
    \end{equation}
    for some positive integer $k=a_1+a_2+\cdots+a_N$
    and $i_1\ge i_2\ge\cdots\ge i_k$.  The \emph{gaps} in the decomposition of $M$
    are the numbers $i_1-i_2,i_2-i_3,\dots,i_{k-1}-i_k$ (for example, $101 = F_{10} + F_5 + F_3 + F_1$, and thus has gaps 5, 2, and 2). Throughout this paper we let $k_\Sigma(M)$ denote the number of summands of $M$
  and $k_g(M)$ the number of gaps of size $g$ in $M$'s decomposition.  Let $K_{\Sigma,n}$ be the random variable equal to
  $k_\Sigma(M)$ for an $M$ chosen uniformly from $[G_n, G_{n+1})$, and let $K_{g,n}$ be a random variable equal to
  $k_g(M)$ for an $M$ chosen uniformly from $[G_n, G_{n+1})$. Thus $k_g(M)$ is a decomposition of $k_\Sigma(M)$,
  as
  \begin{equation}
    k_\Sigma(M) \ = \  1 + \sum_{g=0}^\infty k_g(M).
  \end{equation}

  \begin{theorem}[Generalized Lekkerkerker's Theorem for PLRS]
    \label{thm:summands-lek}
    Let $\{G_n\}$ be a PLRS, let $K_{\Sigma,n}$ be the random variable
    defined above, and let $\mu_n = \E[K_{\Sigma,n}]$.
    Then there exist constants $C_\mu >0$, $d_\mu$, and $\gamma_\mu\in(0,1)$
    depending only on  $L$ and the $c_i$'s of the recurrence relation
    such that \begin{equation}\mu_n \ = \ C_\mu n + d_\mu + O(\gamma_\mu^n).\end{equation}
  \end{theorem}

  \begin{theorem}[Variance is Linear for PLRS]
    \label{thm:summands-var}
    Let $\{G_n\}$ be a PLRS, let $K_{\Sigma,n}$ be the random variable
    defined above, and let $\sigma_n^2 = \Var[K_{\Sigma,n}]$.
    Then there exist constants $C_\sigma >0$, $d_\sigma$, and $\gamma_\sigma\in(0,1)$
    depending only on  $L$ and the $c_i$'s of the recurrence relation
    such that \begin{equation}\sigma_n^2 \ = \ C_\sigma n + d_\sigma + O(\gamma_\sigma^n).\end{equation}
  \end{theorem}

  \begin{theorem}[Gaussian Behavior for Number of Summands in PLRS]
    \label{thm:summands-gauss}
    Let $\{G_n\}$ be a PLRS and let $K_{\Sigma,n}$ be the random variable defined above.
    The mean $\mu_n$ and variance $\sigma_n^2$ of $K_{\Sigma,n}$ grow linearly in $n$,
    and $(K_{\Sigma,n}-\mu_n)/\sigma_n$ converges weakly
    to the standard normal $N(0,1)$ as $n\to\infty$.
  \end{theorem}

Surprisingly, much less has been written on $k_g(M)$ and $K_{g,n}$. We show that similar Central Limit results hold for gaps. The techniques we introduce to prove these results allow us to easily prove some results already in the literature such as the previous three theorems. These proofs are often done through tedious technical calculations, which we can bypass.

  \subsection{New Results}

Beckwith et al. \cite{BBGILMT2013}, Bower et al. \cite{BILMT2015}, and Dorward et al. \cite{DFFHMPP} explored the distribution of gaps
in Generalized Zeckendorf Decompositions arising from PLRS, proving (in the limit $n\to\infty$) exponential decay in the probability that a gap in the decomposition of $M \in [G_n, G_{n+1})$ has length $g$ as $g$ grows and determining that the distribution of the longest gap between summands behaves similarly to what is seen in the distribution of the longest run of heads in tossing a biased coin. We improve on the first result and establish lower order terms (previous work had $O(1)$ instead of $d + o(1)$ below), then prove the variance has a similar linear behavior, and finally show Gaussian behavior for fixed $g$. See \cite{LM} for a similar analysis concentrating on the Fibonacci case, where the simplicity of  the defining recurrence allows simplifications in the analysis.

  \begin{theorem}[Generalized Lekkerkerker's Theorem for Gaps of Decompositions]
    \label{thm:gap-lek}
    Let $g\ge0$ be a fixed positive integer.
    Let $\{G_n\}$ be a PLRS with the additional constraint that all $c_i$s are positive.
    Suppose there exists $n_0\in\NN$
    such that $K_{g,n}$, the random variable defined above,
    is non-trivial (i.e., is not the constant 0) for $n\ge n_0$.
    Let $\mu_{g,n} = \E[K_{g,n}]$.
    Then there exists constants $C_{\mu, g}>0$, $d_{\mu,g}$, and $\gamma_{\mu,g}\in(0,1)$
    depending only on $g$, $L$, and the $c_i$'s of the recurrence relation
    such that \be \mu_{g,n}\ =\ C_{\mu,g} n + d_{\mu,g} + O(\gamma_{\mu,g}^n).\ee
  \end{theorem}

  \begin{theorem}[Variance is Linear for Gaps of Decompositions]
    \label{thm:gap-var}
    Let $g\ge0$ be a fixed positive integer.
    Let $\{G_n\}$ be a PLRS with the additional constraint that all $c_i$s are positive.
    Suppose there exists $n_0\in\NN$
    such that $K_{g,n}$, the random variable defined above,
    is non-trivial for $n\ge n_0$.
    Let $\sigma_{g,n}^2 = \Var[K_{g,n}]$.
    Then there exists constants $C_{\sigma,g}>0$, $d_{\sigma, g}$, and $\gamma_{\sigma,g}\in(0,1)$
    depending only on $g$, $L$, and the $c_i$'s of the recurrence relation
    such that \be \sigma_{g,n}^2\ =\ C_{\sigma,g} n + d_{\sigma, g} + O(\gamma_{\sigma,g}^n).\ee
  \end{theorem}

These two theorems follow as intermediate results in the proof of the next theorem, which is the main result of this paper. The next theorem proves that we also obtain Gaussian behavior if we fix the gap size and if that gap size occurs; note we have to be careful, as there are never gaps of length 1 between summands in Zeckendorf decompositions arising from Fibonacci numbers, and we must make sure to exclude such cases.

  \begin{theorem}[Gaussian Behavior for Gaps of Decompositions]
    \label{thm:gap-gauss}
    Let $g\ge0$ be a fixed positive integer.
    Let $\{G_n\}$ be a PLRS with the additional constraint that all $c_i$s are positive.
    Suppose there exists $n_0\in\NN$
    such that $K_{g,n}$, the random variable defined above,
    is non-trivial for $n\ge n_0$.
    The mean $\mu_{g,n}$ and variance $\sigma_{g,n}^2$ of $K_{g,n}$
    grow linearly in $n$, and $(K_{g,n}-\mu_{g,n})/\sigma_{g,n}$
    converges weakly to the standard normal $N(0,1)$ as $n\to\infty$.
  \end{theorem}

  Our proof uses the fact that $p_{g,n,k}$,
  the number of $M\in[G_n,G_{n+1})$ with exactly $k$ gaps of size $g$,
  satisfies a homogenous two-dimensional recursion
  (see \S\ref{sec:counting-gaps}).
  We then prove that under certain conditions,
  the distributions given by the ``rows'' $\{p_{g,n,k}\}_{k\ge0}$ of these
  two-dimensional homogenous recursions converge to a Gaussian
  (see \S\ref{sec:technical-lemmas}).
  Our proof depends primarily on the recurrence relation, and we only need the initial conditions to be nice enough to avoid edge cases.
  This result should not be surprising,
  as a specific case is the two dimensional recurrence
  $a_{n,k} = a_{n-1,k} + a_{n-1,k-1}$.
  This recurrence produces the binomials $\binom{n}{k}$ (or sums of them), and the distributions formed by binomials $\binom{n}{\cdot}$ are well known to converge to a normal distribution.

  Similar to the work of Miller and Wang \cite{MW1,MW2}, we use the method of moments to
  prove that our random variables converge to Gaussians.
  More precisely, we prove that the moments of
  the $n$\textsuperscript{th} random variable $K_{g,n}$ (or $K_{\Sigma,n}$),
  $\tilde\mu_n(m)$, satisfy
  \begin{equation}
    \lim_{n\to\infty}\frac{\tilde\mu_n(2m)}{\tilde\mu_n(2)^{m}} \ = \  (2m-1)!!
    \quad \text{and} \quad
    \lim_{n\to\infty}\frac{\tilde\mu_n(2m+1)}{\tilde\mu_n(2)^{m+\half}} \ = \  0.
  \end{equation}
  While Miller and Wang use generating functions to directly compute
  the moments $\tilde\mu_n(m)$, we instead compute them recursively (see, for example, Theorem
  \ref{thm:tilde-mu-recursion}), which leads to a cleaner computation and could be of use in other investigations.

\section{Preliminaries}

We first collect some notation we use throughout the paper, then isolate two technical lemmas on convergence, and then apply these to prove Gaussian behavior for certain two dimensional recurrences. This final result is the basis for the proof of our main result on Gaussian behavior of gaps for a fixed $g$, Theorem \ref{thm:gap-gauss}.


\subsection{Notation}
  \label{sec:notation}
    For this paper, all big-Os are taken as $n\to\infty$ unless otherwise specified.

    For a polynomial $A(x) = \sum_{k=0}^d a_kx^k$,
    let $\ba{x^k}(A(x))=a_k$ be the notation
    for extracting the $k$\textsuperscript{th} coefficient of $A$.

    For a real number $\lambda_1>0$, a polynomial $A(x)$ has the
    \textit{maximum root property with maximum root $\lambda_1$}
    if $\lambda_1$ is a root of $A$ with multiplicity 1 and all other roots have
    magnitude strictly less than $\lambda_1$.

    A sequence of real numbers $\{a_n\}$
    \textit{converges exponentially quickly to $a$}
    if $\lim_{n\to\infty}a_n=a$
    and there exists $\gamma\in(0,1)$
    and a constant $C$ such that
    $\abs{a-a_n}\le C\gamma^n$ for all $n$
    (alternatively, $a_n = a + O(\gamma^n)$).

    Let $d$ be a fixed positive integer, and
    let $\{A_n(x)\}$ be a sequence of degree-$d$ polynomials
    where $A_n(x) = \sum_{j=0}^d a_{j,n}x^j$.
    We say $\{A_n(x)\}$
    \textit{converges exponentially quickly to
    $\bar A(x)=\sum_{j=0}^d\bar a_jx^j$}
    if $\{a_{j,n}\}_{n\in\NN}$ converges
    exponentially quickly to $\bar a_j$ for $j=0,1,\dots,d$.

\ \\

From the above definitions we immediately obtain the following useful result.

  \begin{lemma}
    \label{fact:exponential-convergence-arithmetic}
    Let $\{a_n\},\{b_n\}$ be sequences
    that converge exponentially quickly to
    $a$ and $b$ respectively. Then
    \begin{enumerate}
      \ii $\{a_n+b_n\}$ converges exponentially
      quickly to $a+b$,
      \ii $\{a_n-b_n\}$ converges exponentially
      quickly to $a-b$,
      \ii $\{a_n\cdot b_n\}$ converges exponentially
      quickly to $a\cdot b$,
      \ii if $b_n\neq0$ for all $n$
      and $b\neq0$,
      then $\{a_n/ b_n\}$ converges exponentially
      quickly to $a/b$.
    \end{enumerate}
  \end{lemma}

Appendix A of \cite{BBGILMT2013} provides the following useful results.
  \begin{theorem}[Generalized Binet's Formula]
    \label{thm:general-binet}
    Consider any linear recurrence of real numbers (not necessarily a positive linear recurrence)
    \begin{equation}
      G_n\ = \ c_1G_{n-1}+\cdots+c_LG_{n-L}
    \end{equation}
    with arbitrary initial conditions.
    Suppose the characteristic polynomial $x^L-(c_1x^{L-1}+c_2x^{L-2}+\cdots+c_L)$
    has the maximum root property with some maximum root $\lambda_1>0$.
    Then there exists a constant $a_1$ such that
    $G_n = a_1\lambda_1^n+O(n^{L-2}\lambda_2^n)$ where $|\lambda_2| < \lambda_1$ is the second largest root in absolute value.
    Additionally if $a_1$ is positive (that is, $G_n=\Theta(\lambda_1^n)$) then for every fixed positive integer $i$,
    $G_{n-i}/G_n$ converges to $1/\lambda_1^i$ exponentially quickly as $n\to\infty$.
  \end{theorem}

  \begin{theorem}
    \label{thm:plrs-maximum-root}
    \label{lem:plrs-ratio-convergence}
    Consider a PLRS $\{G_n\}$ given by
    \begin{equation}
      G_n\ = \ c_1G_{n-1}+\cdots+c_LG_{n-L}.
    \end{equation}
    Then the characteristic polynomial
    $x^L-(c_1x^{L-1}+c_2x^{L-2}+\cdots+c_L)$
    has the maximum root property with maximum root $\lambda_1>1$
    and $G_n=\Theta(\lambda_1^n)$. In other words, the coefficient $a_1$ given by Theorem \ref{thm:general-binet} is positive.
  \end{theorem}

Note that in Theorem \ref{thm:general-binet}, $a_1$ is positive except for particular choices of initial conditions. For example, if $G_n = 5G_{n-1} - 6G_{n-2}$, we have $G_n = a_13^n + O(2^n)$ unless we have initial conditions $G_1 = \alpha$, $G_2 = 2\alpha$, in which case the $3^n$ term vanishes.



\subsection{Convergence on non-homogenous linear recurrences with noise}\label{sec:technical-lemmas}

The following two lemmas follow immediately from the previous definitions and
book-keeping, and play a key role in the convergence analysis later. In particular, these two lemmas allow us to pin down the exact behavior of the moments of our random variables $K_{g,n}$ as we prove convergence to the standard normal (see Lemmas \ref{lem:mu-calculation} and \ref{lem:moments-calculation}).

    \begin{lemma}
      \label{lem:key-technical}
      Let $i_0$ be a positive integer.
      Let $\{r_{n}\}_{n\in\NN}$ be a sequence of real numbers and for each $1\le i\le i_0$ let $\{s_{i,n}\}_{n\in\NN}$ be a sequence of non-negative real numbers such that $\sum_{i=1}^{i_0} s_{i,n}=1$ for all $n$.
      With slight abuse of notation, suppose also that there exist constants $\bar r$ and $\bar s_i$ for $1\le i\le i_0$, along with $\gamma_r,\gamma_s\in(0,1)$ such that
      \begin{align}
        r_{n} \ = \  \bar r + O(\gamma_r^n), &\qquad s_{i,n} \ = \  \bar s_i + O(\gamma_s^n).
      \end{align}
      Suppose further that the polynomial
      \begin{equation}
        \label{eq:ci-characteristic-polyn}
        S(x) \ = \  x^{i_0} - \sum_{i=1}^{i_0} \bar s_i x^{i_0-i}
      \end{equation}
      has the maximum root property with maximum root 1.
      Let $\{a_n\}_{n\ge n_0}$ be a sequence with arbitrary initial conditions $a_{n_0},\dots,a_{n_0+i_0-1}$ and, for $n \ge n_0 + i_0$,
      \begin{equation}
        a_n \ = \  \pa{\sum_{i=1}^{i_0}s_{i,n}a_{n-i}}+r_{n}.
      \end{equation}
      Then there exists a positive integer $d$ and a real number $\gamma\in(0,1)$ such that
      \begin{equation}
        a_n \ = \  \frac{\bar r}{\sum_{i=1}^{i_0}i\cdot\bar s_i}\cdot n+d+O(\gamma^n).
      \end{equation}
    \end{lemma}
    Roughly speaking, Lemma \ref{lem:key-technical} is true because, modulo exponentially small terms, every $a_n$ is a constant plus the weighted average of previous $a_{n-i}$s, so it should be linear in $n$.
    \begin{proof}
      It suffices to prove the lemma for $n_0=0$.
      Let $b_n = a_n - \frac{\bar r}{\sum_{i=1}^{i_0}i\cdot\bar s_i}\cdot n$.
      Set $\gamma=\max(\gamma_r,\gamma_s)$.
      Simple manipulations yield
      \begin{align}
        b_n
          &\ = \  \sum_{i=1}^{i_0} s_{i,n}b_{n-i}
          + r_n - \frac{\sum_{i=1}^{i_0}i\cdot s_{i,n}}
                       {\sum_{i=1}^{i_0} i\cdot\bar s_i}\cdot \bar r \nonumber\\
          &\ = \  \sum_{i=1}^{i_0} s_{i,n}b_{n-i}
            + \bar r \cdot
            \pa{\frac{r_n}{\bar r}
              - \frac{\sum_{i=1}^{i_0}i\cdot\bar s_{i,n}}{\sum_{i=1}^{i_0} i\cdot\bar s_i}}
                        \nonumber\\
          &\ = \  \sum_{i=1}^{i_0} s_{i,n}b_{n-i}
            + \bar r \cdot \pa{(1+O(\gamma^n)) - (1+O(\gamma^n))} \nonumber\\
          &\ = \  \sum_{i=1}^{i_0} s_{i,n}b_{n-i} + O(\gamma^n).
      \end{align}

      We finish by showing that the sequence $b_n$ converges
      exponentially quickly to a constant.
      Simple algebra yields that $b_n$ is bounded (see Appendix \ref{app:bn-bound}).
      Let $M$ be an integer such that $|b_n|\le M$ for all $n$.
      Then
      \begin{align}
        b_n-\sum_{i=1}^{i_0}\bar s_ib_{n-i}
          &\ = \  b_n-\sum_{i=1}^{i_0}s_{i,n}b_{n-i}
            + \sum_{i=1}^{i_0}(s_{i,n}-\bar s_i)b_{n-i} \nonumber\\
          &\ \le \  O(\gamma^n) + \sum_{i=1}^{i_0}O(\gamma^n)\cdot b_{n-i} \nonumber\\
          &\ \le \  O(\gamma^n) + \sum_{i=1}^{i_0}O(\gamma^n)\cdot M \ = \  O(\gamma^n).
      \end{align}
      Thus we can write
      \begin{equation}
        b_n\ = \ \pa{\sum_{i=1}^{i_0}\bar s_ib_{n-i}} + f(n)
      \end{equation}
      for some function $f:\{i_0,i_0+1,\dots\}\to\RR$
      such that $f(n) = O(\gamma^n)$ as $n\to\infty$.
      Let $\alpha_f>0$ be a constant such that $|f(n)|\le \alpha_f\gamma^n$.

      From here, the intuition for the finish is as follows.
      If $f(n)=0$ for all $n$, then 
      Theorem \ref{thm:general-binet}
      implies that $b_n$ approaches a constant exponentially quickly.
      However, since $\gamma<1$,
      we have that $b_n$ should still approach
      a constant exponentially quickly when $f(n)=O(\gamma^n)$.

      Let $\{b\id{\init}_n\}_{n\in\NN},
      \{b\id{i_0}_n\}_{n\in\NN}, \{b\id{i_0+1}_n\}_{n\in\NN}, \dots$
      be sequences defined (for $m\ge i_0$) by
      \begin{align}
        b\id{\init}_n &\ = \
          \left\{
            \begin{array}{cl}
              b_n& 0\le n\le i_0-1\\
              \sum_{i=1}^{i_0}\bar s_i b\id{\init}_{n-i} & n>i_0\\
            \end{array}
          \right. \nonumber\\
        b\id m_n &\ = \
          \left\{
            \begin{array}{cl}
              0& n<m\\
              f(m)& n=m\\
              \sum_{i=1}^{i_0}\bar s_i b\id m_{n-i} & n>m.\\
            \end{array}
          \right.
      \end{align}
      By induction, we can verify that
      \begin{equation}
        b_n \ = \  b\id{\init}_n+\sum_{m=i_0}^\infty b\id{m}_n
      \end{equation}
      for all $n$ (see Appendix \ref{app:bn-decomposition}).
      By the restrictions of $s_i$, the characteristic
      polynomials of $\{b\id{\init}_n\}$ and $\{b\id{m}_n\}$
      are equal to $S(x)$ in \eqref{eq:ci-characteristic-polyn}
      and thus have the maximum root property with maximum root 1.
      Hence, by generalized Binet's formula $\{b\id{\init}_n\}$ and $\{b\id{m}_n\}$ all converge to a constant.
      Suppose that $\{b\id{\init}_n\}$ converges to $\bar b\id{\init}$
      and $\{b\id{m}_n\}$ converges to $\bar b\id{m}$ for each $m\ge i_0$.
      Let $\lambda_2<1$ be the second largest magnitude of a root of $S(x)$.
      Choose $\gamma_1\in(\max(\gamma,\lambda_2),1)$.
      By the generalized Binet's formula
       \begin{equation}
        b_n\id{\init}-\bar b\id{\init} \ = \  O(n^{i_0}\cdot\lambda_2^n) \ \le \  O(\gamma_1^n),
      \end{equation}
      so there exists $\alpha_{b}\id1$ such that
      \begin{equation}
        \abs{b_n\id{\init}-\bar b\id{\init}} \ \le \  \alpha_b\id1\gamma_1^n.
      \end{equation}
      For all $m$, we can bound $b\id m_n$ similarly.
      However, note that for all $m$,
      $\{b\id m_n/f(m)\}_{n\in\NN}$ is
      the same sequence with the indices shifted.
      Thus there exists $\alpha_b\id2$ such that
      \begin{equation}
        \abs{b_n\id{m}-\bar b\id{m}}
          \ \le \  \alpha_b\id2f(m)\gamma_1^{n-m}
          \ \le \ \alpha_b\id2\alpha_f\gamma^m\gamma_1^{n-m}.
      \end{equation}
      Set $\alpha_b=\max(\alpha_b\id1,\alpha_b\id2)$. Then
      \begin{align}
        \abs{b_n - b}
          &\ \le \
          \abs{b\id{\init}-b\id{\init}_n}+\sum_{m=i_0}^{\infty}\abs{b\id m-b\id m_n}
          \nonumber\\
          &\ \le \  \alpha_b\gamma_1^n
            + \sum_{m=i_0}^\infty\alpha_b\alpha_f\gamma_1^n\pa{\frac{\gamma}{\gamma_1}}^m
            \nonumber\\
          &\ \le \  \gamma_1^n\pa{\alpha_b
            + \alpha_b\alpha_f\cdot\pa{\frac{\gamma}{\gamma_1}}^{i_0}
                \cdot\frac{1}{1-\frac{\gamma}{\gamma_1}}} \ = \  O(\gamma_1^n)
      \end{align}
      as desired.
    \end{proof}

    The next lemma generalizes Lemma \ref{lem:key-technical}.
    \begin{lemma}
      \label{lem:key-technical-polyn}
      Let $D$ be a nonnegative integer and let $i_0$ be a positive integer.
      Let $\{R_{n}(x)\}_{n\in\NN}$ be a sequence of $D$ degree polynomials
      with $R_n(x) = \sum_{j=0}^Dr_{j,n}x^j$.
      For each $1\le i\le i_0$
      let $\{s_{i,n}\}_{n\in\NN}$ be a sequence of non-negative
      real numbers such that $\sum_{i=1}^{i_0} s_{i,n}=1$ for all $n$.
      Suppose also that there exist a polynomial
      $\bar R(x)=\bar r_Dx^D+\bar r_{D-1}x^{D-1}+\cdots+\bar r_0$
      and real numbers $\bar s_i$ for $1\le i\le i_0$,
      along with $\gamma_r,\gamma_s\in(0,1)$, such that,
      for all $0\le j\le D$ and $1\le i\le i_0$,
      \begin{align}
        r_{j,n} \ = \  \bar r_j + O(\gamma_r^n),\qquad
        s_{i,n} \ = \  \bar s_i + O(\gamma_s^n).
      \end{align}
      Suppose further that the polynomial,
      $S(x) = x^{i_0} - \sum_{i=1}^{i_0} \bar s_i x^{i_0-i}$
      has the maximum root property with maximum root 1.
      Let $\{a_n\}_{n\ge n_0}$ be a sequence with arbitrary initial conditions
      $a_{n_0}$, $\dots$, $a_{n_0+i_0-1}$, and for $n \ge n_0 + i_0$,
      \begin{equation}
        \label{eq:key-technical-polyn-an}
        a_n \ = \  \pa{\sum_{i=1}^{i_0}s_{i,n}a_{n-i}}+R_{n}(n).
      \end{equation}
      Then there exists a degree $D+1$ polynomial $Q(x)$ and a
      $\gamma_1\in (0,1)$ such that
      \begin{equation}
        a_n \ = \  Q(n) + O(\gamma_1^n)
      \end{equation}
      where
      \begin{equation}
        \ba{x^{D+1}}(Q(x)) \ = \  \frac{\bar r_D}{(D+1)\sum_{i=1}^{i_0} i\cdot\bar s_i}.
      \end{equation}
    \end{lemma}
    In contrast to in Lemma \ref{lem:key-technical}, $a_n$ is, modulo exponentially small terms, a $D$ degree polynomial in $n$ plus the weighted average of previous $a_{n-i}$s. Since for any $D$ degree polynomial $A(x)$ the sum $A(1)+A(2)+\cdots+A(n)$ is an $D+1$ degree polynomial in $n$, we expect $a_n$ to essentially be a $D+1$ degree polynomial in $n$.
    \begin{proof}
      We proceed by induction on $D$, the degree of the polynomials $R_n(x)$.
      Lemma \ref{lem:key-technical} establishes the base case $D=0$.
      Now assume $D>0$ and that the assertion is true for $D-1$.
      Let $b_n = a_n - C\cdot n^{D+1}$
      for $C=\frac{\bar r_D}{(D+1)\sum_{i=1}^{i_0} i\cdot \bar s_i}$.
      Straightforward manipulations yield
      \begin{align}
        b_n
          \ = \  \sum_{i=1}^{i_0}s_{i,n}b_{n-i}
             + \sum_{j=0}^{D-1}n^j\cdot
                \ba{\pa{
                    \sum_{i=1}^{i_0}C s_{i,n}(-1)^{D+1-j}\binom{D+1}{j}i^{D+1-j}}
                    + r_{j,n}}
                + f(n)
      \end{align}
      for some function $f(n)\le O(\gamma_0^n)$ for some $\gamma_0\in(0,1)$ (see Appendix \ref{app:bn-simplify}). The constant $C$ is chosen so that the right side contains an $D-1$ degree polynomial in $n$, as opposed to an $D$ degree polynomial, which is the case in the recursion for $\{a_n\}$.
      Let $R_n^*(x) = \sum_{j=0}^{D-1}r_{j,n}^*x^j$ be the polynomial given by
      \begin{align}
        r_{0,n}^*
        &\ \defeq\
          \pa{\sum_{i=1}^{i_0}C s_{i,n}(-1)^{D+1}\binom{D+1}{0}i^{D+1}} + r_{0,n} + f(n)
          \nonumber\\
        r_{j,n}^*
        &\ \defeq\
          \pa{\sum_{i=1}^{i_0}C s_{i,n}(-1)^{D+1-j}\binom{D+1}{j}i^{D+1-j}} + r_{j,n}
      \end{align}
      for $1\le j\le {D-1}$.
      Since, as $n\to\infty$, $s_{i,n}$ converges exponentially fast to $\bar s_i$, $r_{j,n}$ converges exponentially fast to $\bar r_j$, and $f(n)$ converges exponentially fast to 0, we have $r_{j,n}^*$ converges exponentially fast to
      \begin{equation}
        \lim_{n\to\infty} r_{j,n}^*
          \ = \  \pa{\sum_{i=1}^{i_0}C \bar s_i(-1)^{D+1-j}\binom{D+1}{j}i^{D+1-j}} + \bar r_j.
      \end{equation}
      for $0\le j\le D-1$ by Lemma \ref{fact:exponential-convergence-arithmetic}.
      Writing
      \begin{align}
        b_n
          \ = \  \pa{\sum_{i=1}^{i_0}s_{i,n}b_{n-i}} + R^*_n(n),
      \end{align}
      we can apply the induction hypothesis to $b_n$ to obtain a
      a degree $D$ polynomial $Q^*(x)$
      such that $b_n = Q^*(n) + O(\gamma_1^n)$ for some $\gamma_1\in(0,1)$.
      Set $Q(x) = Q^*(x) + Cx^{D+1}$. Then $Q(x)$ is a degree $D+1$ polynomial
      satisfying $a_n = Q(x) + O(\gamma_1^n)$, as desired.
    \end{proof}


  \subsection{Gaussian Behavior of 2D Recursions}\label{sec:moments}

The result below is the key ingredient in proving Gaussian behavior of gaps.

  \begin{theorem}
    \label{thm:2d-recursion}
    Let $i_0$ and $j_0$ be positive integers.
    Let $t_{i,j}$ be real numbers for $1\le i\le i_0, 0\le j\le j_0$
    such that for all $i$, $\hat t_i \defeq \sum_{j=0}^{j_0} t_{i,j}\ge 0$.
    Suppose that the polynomial
    $T(x)=x^{i_0}-\sum_{i=1}^{i_0}\hat t_ix^{i_0-i}$
    has the maximum root property with some maximum root $\lambda_1$.
    Suppose $p_{n,k}$ is a two-dimensional recurrence sequence satisfying,
    for $n\ge n_0$,
    \begin{equation}
      \label{eq:R-def}
      p_{n,k}\ = \ \sum_{i=1}^{i_0}\sum_{j=0}^{j_0}t_{i,j}p_{n-i,k-j}.
    \end{equation}
    Suppose further that $p_{n,k}\ge 0$ for all $n$ and $k$, $p_{n,k}=0$ when $n<0$ or $k<0$, finitely many $p_{n,k}$ are nonzero for $n<n_0$, and $\sum_{i=0}^\infty p_{n,i}=\Theta(\lambda_1^n)$.
    Let $X_n$ be the random variable whose mass function is proportional to $p_{n,k}$ over varying $k$ so that
    \begin{equation}
      \Pr[X_n=k] \ = \  \frac{p_{n,k}}{\sum_{i=0}^\infty p_{n,i}}.
    \end{equation}
    Let
    \begin{equation}
      \label{eq:C_mu-and-C_sigma}
      C_\mu \ = \  \frac{\sum_{i=1}^{i_0}\sum_{j=0}^{j_0} \frac{t_{i,j}\cdot j}{\lambda_1^i}}
                   {\sum_{i=1}^{i_0}\sum_{j=0}^{j_0} \frac{t_{i,j}\cdot i}{\lambda_1^i}},
        \qquad
      C_\sigma \ = \  \frac
            {\sum_{i=1}^{i_0}\sum_{j=0}^{j_0}
            \frac{t_{i,j}}{\lambda_1^i}
            \cdot(j-C_\mu i)^2 }
            {\sum_{i=1}^{i_0}\sum_{j=0}^{j_0}\frac{t_{i,j}\cdot i}{\lambda_1^i}}
    \end{equation}
    be constants, and assume $C_\sigma>0$.
    Then there exist constants $d_\mu, d_\sigma\in\RR,$ and $\gamma_\mu,\gamma_\sigma\in(0,1)$ such that $\mu_n = C_\mu n + d_\mu + O(\gamma_\mu^n)$ and $\sigma_n^2 = C_\sigma n + d_\sigma + O(\gamma_\sigma^n)$.
    Furthermore, $(X_n-\mu_n)/\sigma_n$ converges weakly to the standard normal $N(0,1)$ as $n\to\infty$.
  \end{theorem}
  In this theorem, imagine we have fixed a gap size $g$ and think of $p_{n,k}$ as the number of $M\in[G_n, G_{n+1})$ whose decomposition has exactly $k$ gaps of size $g$. Under this interpretation, the random variable $X_n$ is be identical to $K_{g,n}$.
  Note that $T(x)$ having the maximum root property does not make the condition $\sum_{i=0}^\infty p_{n,i}=\Theta(\lambda_1^n)$ redundant for reasons illustrated at the end of \S\ref{sec:notation}. This condition is necessary in Corollary \ref{lem:omega-recursion}.

  We approach this problem using the method of moments, a common method for proving random variables converge in distribution to the standard normal distribution.
  \begin{lemma}[Method of Moments]
    \label{lem:moments}
    Suppose $X_1, X_2,\dots$ are random variables such that for all integers $m\ge 0$, we have
    \begin{equation}
      \lim_{n\to\infty}\E[X_n^{2m}] \ = \  (2m-1)!!
      \quad
      \text{and}
      \quad
      \lim_{n\to\infty}\E[X_n^{2m+1}] \ = \  0.
    \end{equation}
    Then the sequence $X_1, X_2,\dots$ converges weakly in distribution to the standard normal $N(0,1)$.
  \end{lemma}

    The proof of Theorem \ref{thm:2d-recursion} proceeds by using generating functions to compute the moments of $X_n$.  Let
    \begin{align}
      P_n(x) &\ \defeq\  \sum_{k=0}^\infty p_{n,k}x^k \nonumber\\
      \Omega_n &\ \defeq\  P_n(1) = \sum_{k=0}^\infty p_{n,k} \nonumber\\
      \tilde P_{n,0}(x) &\ \defeq\  \frac{P_n(x)}{x^{\mu+1}} \nonumber\\
      \tilde P_{n,m}(x) &\ \defeq\  (x\tilde P_{n,m-1}(x))' \nonumber\\
      \tilde\mu_n(m) &\ \defeq\  \frac{\tilde P_{n,m}(1)}{\Omega_n}.
    \end{align}
    Then it follows from definitions that
    \begin{align}
      \mu_n &\ = \  \frac{P_n'(1)}{P_n(1)} \nonumber\\
      \tilde\mu_n(m) &\ = \  \E[(X_n-\mu_n)^m] \nonumber\\
      \sigma_n^2 &\ = \  \tilde\mu_n(2).
    \end{align}
    We now prove several lemmas about the above moments and generating functions.
    We ultimately obtain a formula in Theorem \ref{thm:tilde-mu-recursion} that recursively computes the moments $\tilde\mu_n(m)$, which yield Theorem \ref{thm:2d-recursion}.
    \begin{lemma}
      \label{lem:p-recursion}
      For $n\ge n_0$, we have
      \begin{equation}
        \label{eq:p-recursion}
        P_n(x) \ = \  \sum_{i=1}^{i_0}\sum_{j=0}^{j_0} t_{i,j}P_{n-i}(x)x^j.
      \end{equation}
    \end{lemma}
    \begin{proof}
      This follows immediately from the definitions:
      \begin{align}
        P_n(x)
          &\ = \  \sum_{k=0}^\infty p_{n,k}x^k \nonumber\\
          &\ = \  \sum_{k=0}^\infty \sum_{i=1}^{i_0}\sum_{j=0}^{j_0} t_{i,j}p_{n-i,k-j}x^k \nonumber\\
          &\ = \  \sum_{i=1}^{i_0}\sum_{j=0}^{j_0}t_{i,j}x^j\sum_{k=0}^\infty p_{n-i,k-j}x^{k-j} \nonumber\\
          &\ = \  \sum_{i=1}^{i_0}\sum_{j=0}^{j_0}t_{i,j}x^jP_{n-i}(x).
      \end{align}
    \end{proof}

    From the above we immediately deduce the following relations.

    \begin{corollary}
      \label{lem:omega-recursion}
      \label{lem:mu-recursion}
      For $n\ge n_0$, we have
      \begin{equation}
        \Omega_n
          \ = \  P_n(1)
          \ = \  \tilde P_{n,0}(1)
          \ = \  \sum_{i=1}^{i_0}\sum_{j=0}^{j_0}t_{i,j}P_{n-i}(1)
          \ = \  \sum_{i=1}^{i_0}\sum_{j=0}^{j_0}t_{i,j}\Omega_{n-i}.
      \end{equation} and
      \begin{equation}
        \label{eq:mu-recursion}
        \mu_n
          \ = \  \sum_{i=1}^{i_0}\sum_{j=0}^{j_0} \frac{t_{i,j}\Omega_{n-i}}{\Omega_n}(\mu_{n-i}+j).
      \end{equation}
      By definition of $\Omega_n$, we have $\Omega_n=\Theta(\lambda_1^n)$, so by Theorem \ref{thm:general-binet}, we have that for all $i$, $\Omega_{n-i}/\Omega_n$ converges exponentially quickly to $1/\lambda_1^i$.
    \end{corollary}

    \begin{theorem}
      \label{thm:tilde-mu-recursion}
      For $n\ge n_0$, we have
      \begin{equation}
        \label{eq:tilde-mu-recursion}
        \tilde\mu_n(m)
          \ = \  \sum_{\ell=0}^m\binom{m}{\ell}
            \sum_{i=1}^{i_0}\sum_{j=0}^{j_0}
              \frac{\Omega_{n-i}t_{i,j}}{\Omega_n}
              \cdot (j+\mu_{n-i}-\mu_n)^\ell
              \cdot \tilde\mu_{n-i}(m-\ell).
      \end{equation}
    \end{theorem}
    \begin{proof}
      Applying Lemma \ref{lem:p-recursion}, we find
      \begin{equation}
        \label{eq:tilde-P_n0}
        \tilde P_{n,0}(x)
          \ = \  \sum_{i=1}^{i_0}\sum_{j=0}^{j_0}t_{i,j}
            \tilde P_{n-i,0}(x)\cdot x^{j+\mu_{n-i}-\mu_n}.
      \end{equation}
      By induction, we can establish (see Appendix \ref{app:tilde-P_nm})
      \begin{equation}
        \label{eq:tilde-P_nm}
        \tilde P_{n,m}(x)
          \ = \  \sum_{i=1}^{i_0}\sum_{j=0}^{j_0}t_{i,j}
            \sum_{\ell=0}^m
              \binom{m}{\ell}(j+\mu_{n-i}-\mu_n)^\ell
                \tilde P_{n-i,m-\ell}(x)\cdot x^{j+\mu_{n-i}-\mu_n},
      \end{equation}
      so
      \begin{align}
        \tilde\mu_n(m)
          \ = \  \frac{\tilde P_{n,m}(1)}{\Omega_n}
          &\ = \  \frac{1}{\Omega_n}\sum_{i=1}^{i_0}\sum_{j=0}^{j_0}t_{i,j}
            \sum_{\ell=0}^m
              \binom{m}{\ell}(j+\mu_{n-i}-\mu_n)^\ell
                \tilde P_{n-i,m-\ell}(1) \nonumber\\
          &\ = \  \frac{1}{\Omega_n}
            \sum_{\ell=0}^m\binom{m}{\ell}
            \sum_{i=1}^{i_0}\sum_{j=0}^{j_0}t_{i,j}
              (j+\mu_{n-i}-\mu_n)^\ell
                \Omega_{n-i}\tilde\mu_{n-i}(m-\ell) \nonumber\\
          &\ = \  \sum_{\ell=0}^m\binom{m}{\ell}
            \sum_{i=1}^{i_0}\sum_{j=0}^{j_0}
              \frac{\Omega_{n-i}t_{i,j}}{\Omega_n}
              \cdot (j+\mu_{n-i}-\mu_n)^\ell
              \cdot \tilde\mu_{n-i}(m-\ell),
      \end{align} completing the proof.
    \end{proof}

    Out next goal is to prove.
    \begin{equation}
      \label{eq:moments-gaussian-condition}
      \lim_{n\to\infty}\frac{\tilde\mu_n(2m)}{\tilde\mu_n(2)^{m}} \ = \  (2m-1)!!
      \quad
      \text{and}
      \quad
      \lim_{n\to\infty}\frac{\tilde\mu_n(2m+1)}{\tilde\mu_n(2)^{m+\half}} \ = \  0.
    \end{equation}
    By Lemma \ref{lem:moments}, these equalities imply Theorem \ref{thm:2d-recursion}.
    To prove these equalities, we first show $\mu_n$ is essentially linear in $n$. Then we determine for all $m$ the behavior of $\tilde\mu_n(m)$, the $m$\tsup{th} moment of $X_n-\mu_n$, up to an exponentially small term. We prove $\tilde\mu_n(m)$ is a degree (at most, if $m$ is odd) $\floor{m/2}$ polynomial in $n$, and for even moments $\tilde\mu_n(2m)$ we addtionally compute the leading coefficient of this polynomial. We rely heavily on Lemmas \ref{lem:key-technical} and \ref{lem:key-technical-polyn} to pin down the polynomial behavior of the moments.

    \begin{lemma}
      \label{lem:mu-calculation}
      There exists a real number $d_\mu$ and a $\gamma_\mu\in(0,1)$ such that
      \begin{equation}
        \label{eq:mu-calculation}
        \mu_n \ = \  C_\mu\cdot n + d_\mu + O(\gamma_\mu^n).
      \end{equation}
    \end{lemma}

    \begin{proof}
      Recall
      \be
        C_\mu \ = \  \frac{\sum_{i=1}^{i_0}\sum_{j=0}^{j_0}\frac{t_{i,j}\cdot j}{\lambda_1^i}}
                     {\sum_{i=1}^{i_0}\sum_{j=0}^{j_0}\frac{t_{i,j}\cdot i}{\lambda_1^i}}.
      \ee
      Choose $s_{i,n} = \frac{\Omega_{n-i}}{\Omega_n}\sum_{j=0}^{j_0}t_{i,j} = \frac{\Omega_{n-i}}{\Omega_n}\hat t_i$
      and $r_{n} = \sum_{i=1}^{i_0}\sum_{j=0}^{j_0}\frac{t_{i,j}\cdot j\cdot \Omega_{n-i}}{\Omega_n}$.
      Using Lemmas \ref{fact:exponential-convergence-arithmetic}
      and \ref{lem:omega-recursion},
      we have that, for each $i$, $s_{i,n}$ converges exponentially quickly to
      $\bar s_i=\frac{1}{\lambda_1^i}\sum_{j=0}^{j_0}t_{i,j} = \hat t_i$
      and $r_n$ converges exponentially quickly to
      $\bar r = \sum_{i=1}^{i_0}\sum_{j=0}^{j_0}\frac{t_{i,j}\cdot j}{\lambda_1^i}$.
      By Lemma \ref{lem:mu-recursion}, we have
      \be
        \mu_n \ = \  \pa{\sum_{i=1}^{i_0}s_{i,n}\mu_{n-i}} + r_n.
      \ee
      Furthermore, the polynomial
      $S(x) \defeq x^{i_0} - \sum_{i=1}^{i_0}\bar s_i x^{i_0-i}$
      satisfies $S(x) = T(x/\lambda_1)$,
      so $S$ has the maximum root property with maximum root 1.
      Then, by Lemma \ref{lem:key-technical},
      there exist $d_\mu\in\RR$ and $\gamma_\mu\in(0,1)$ such that
      \begin{equation}
        \mu_n
          \ = \ \frac{\bar r}{\sum_{i=1}^{i_0}i\bar s_i} \cdot n + d_\mu + O(\gamma_\mu^n)
          \ = \ C_\mu\cdot n + d_\mu + O(\gamma_\mu^n).
      \end{equation}
    \end{proof}
    \begin{lemma}
      \label{lem:moments-calculation}
      For each integer $m\ge0$, there exist
      $\gamma_{2m},\gamma_{2m+1}\in(0,1)$
      and polynomials $Q_{2m}$ of degree exactly $m$
      and $Q_{2m+1}$ of degree at most $m$
      such that
      \begin{align}
        \tilde\mu_n(2m) &\ = \  Q_{2m}(n) + O(\gamma_{2m}^n) \nonumber\\
        \tilde\mu_n(2m+1) &\ = \  Q_{2m+1}(n) + O(\gamma_{2m+1}^n).
      \end{align}
      Furthermore, if $C_{2m}\defeq[x^m]Q_{2m}$ and $C_{2m+1}\ \defeq\ [x^m]Q_{2m+1}$, then for all $m\ge0$, $C_{2m}$ $=$ $(2m-1)!!\cdot C_\sigma^m$ (We take $(-1)!!\defeq 1$).
    \end{lemma}
    The idea for the proof is as follows. In the calculation of $\mu_n(m)$ in Theorem \ref{thm:tilde-mu-recursion} the coefficients of $\mu_{n-i}(m)$ sum to 1, the coefficients of $\mu_{n-i}(m-1)$ sum to 0, and the coefficients of $\mu_{n-i}(m-2)$ sum to $\binom{m}{2}\cdot\text{(constant)}$.
    The $m$\tsup{th} moments can thus be written in the form of \eqref{eq:key-technical-polyn-an}, so we can apply Lemma \ref{lem:key-technical-polyn} and compute the degrees and leading coefficients appropriately. Because the coefficients of the $(m-1)$\tsup{th} moments sum to 0, the degrees of the polynomials increase by one with every two values of $m$ as opposed to every one.
    \begin{proof}
      We proceed by induction on $m$.
      The base case $m=0$ follows from noting that
      \begin{align}
        \tilde\mu_n(0) &\ = \  \E[(X_n-\mu_n)^0]\ = \ 1 \nonumber\\
        \tilde\mu_n(1) &\ = \  \E[(X_n-\mu_n)^1]\ = \ 0
      \end{align}
      for all $n\ge n_0$.
      Now assume the statement is true for $m'\le m$.
      That is, there exist
      $\gamma_0$, $\gamma_1$, $\dots$, $\gamma_{2m-1}$ $\in$ $(0,1)$
      and polynomials $Q_0, Q_1, \dots, Q_{2m-1}$ where $Q_k$ has degree $k/2$ when $k$ is even and degree at most $\floor{k/2}$ when $k$ is odd
      such that
      \begin{eqnarray}
        \tilde\mu_n(2m-2) & \ = \ &  Q_{2m-2}(n) + O(\gamma_{2m-2}^n) \nonumber\\
        \tilde\mu_n(2m-1) & \ = \ & Q_{2m-1}(n) + O(\gamma_{2m-1}^n).
      \end{eqnarray}
      By induction we may assume further that $C_{2m-2}=(2m-3)!!\cdot C_\sigma^{m-1}$.
      First, we compute $\tilde\mu_n(2m)$.
      Define a sequence of polynomials $\{R_n(x)\}$ via
      \begin{equation}
        \label{eq:tilde-beta-def}
        R_n(x)
          \ \defeq\  \sum_{\ell=1}^{2m}\sum_{i=1}^{i_0}\sum_{j=0}^{j_0}
            \frac{\Omega_{n-i}t_{i,j}}{\Omega_n}
            \cdot(j+\mu_{n-i}-\mu_n)^\ell
            \cdot Q_{2m-\ell}(x-i).
      \end{equation}
      Furthermore, set
      \begin{equation}
        s_{i,n}\ = \ \frac{\Omega_{n-i}}{\Omega_n}\sum_{j=0}^{j_0}t_{i,j}
        \quad\text{and}\quad
        \bar s_i \ = \  \frac{1}{\lambda_1^i} \sum_{j=0}^{j_0}t_{i,j}.
      \end{equation}
      Then
      \begin{align}
        \tilde\mu_n(2m)
          \ = \  \sum_{i=1}^{i_0}s_{i,n}\tilde\mu_{n-i}(2m) + R_n(n).
      \end{align}
      Note that $R_n(x)$ is the sum of finitely many
      polynomials that,
      by Lemma \ref{fact:exponential-convergence-arithmetic},
      converges exponentially quickly.
      Thus $R_n(x)$ converges exponentially quickly to
      \begin{equation}
        \label{eq:bar-R-definition}
        \bar R(x)
          \ \defeq\  \sum_{\ell=1}^{2m}\sum_{i=1}^{i_0}\sum_{j=0}^{j_0}
            \frac{t_{i,j}}{\lambda_1^i}
            \cdot(j-C_\mu i)^\ell
            \cdot Q_{2m-\ell}(x-i).
      \end{equation}
      Furthermore, we have $\deg \bar R(x)\le m-1$
      since each $R_n(x)$ has degree at most $m-1$.
      We can compute the leading coefficient of $R$
      using \eqref{eq:bar-R-definition} to get
      \begin{align}
        \label{eq:tilde-beta-leading-coefficient}
        [x^{m-1}](\bar R(x))
          &\ = \  \sum_{\ell=1}^{2m}\sum_{i=1}^{i_0}\sum_{j=0}^{j_0}
            \binom{2m}{\ell}
            \frac{t_{i,j}}{\lambda_1^i}
            \cdot(j-C_\mu i)^\ell
            \cdot [x^{m-1}](Q_{2m-\ell}(x-i)) \nonumber\\
          &\ = \  \sum_{\ell=1}^{2}\sum_{i=1}^{i_0}\sum_{j=0}^{j_0}
            \binom{2m}{\ell}
            \frac{t_{i,j}}{\lambda_1^i}
            \cdot(j-C_\mu i)^\ell
            \cdot C_{2m-\ell}
            \nonumber\\
          &\ = \
            C_{2m-2} \cdot \binom{2m}{2} \sum_{i=1}^{i_0}\sum_{j=0}^{j_0} \frac{t_{i,j}}{\lambda_1^i} \cdot(j-C_\mu i)^2
            + C_{2m-1} \cdot 2m\sum_{i=1}^{i_0}\sum_{j=0}^{j_0} \frac{t_{i,j}}{\lambda_1^i} \cdot(j-C_\mu i)
            \nonumber\\
          &\ = \
            C_{2m-2} \cdot \binom{2m}{2} \sum_{i=1}^{i_0}\sum_{j=0}^{j_0} \frac{t_{i,j}}{\lambda_1^i} \cdot(j-C_\mu i)^2
            + C_{2m-1} \cdot 2m\cdot 0 \nonumber\\
          &\ = \
            C_{2m-2} \cdot\binom{2m}{2} \sum_{i=1}^{i_0}\sum_{j=0}^{j_0} \frac{t_{i,j}}{\lambda_1^i} \cdot(j-C_\mu i)^2.
      \end{align}
      Recalling the definition
      \begin{equation}
        C_\sigma
          \ = \ \frac{\sum_{i=1}^{i_0}\sum_{j=0}^{j_0}
            \frac{t_{i,j}}{\lambda_1^i}
            \cdot(j-C_\mu i)^2 }
            {\sum_{i=1}^{i_0}\sum_{j=0}^{j_0}\frac{t_{i,j}\cdot i}{\lambda_1^i}}
          \ = \ \frac{\sum_{i=1}^{i_0}\sum_{j=0}^{j_0}
            \frac{t_{i,j}}{\lambda_1^i}
            \cdot(j-C_\mu i)^2 }
            {\sum_{i=1}^{i_0}i\cdot \bar s_i},
      \end{equation}
      we have
      \begin{align}
        [x^{m-1}](\bar R(x))
          &\ = \ C_{2m-2} \cdot \binom{2m}{2}\cdot C_\sigma\cdot \pa{\sum_{i=1}^{i_0}i\cdot\bar s_i}.
      \end{align}
      By Lemma \ref{lem:key-technical-polyn}, there exists a degree $\deg\bar R(x) + 1$ polynomial $Q_{2m}(x)$ with $x^m$ coefficient $C_{2m}$ and a $\gamma_{2m}\in(0,1)$ such that
      \begin{align}
        \mu_n(2m) &\ = \  Q_{2m}(n) + O(\gamma_{2m}^n)
      \end{align}
      and
      \begin{align}
        \label{eq:C_2m-calculation}
        C_{2m}
          \ = \
            \frac{C_{2m-2} \cdot \binom{2m}{2}\cdot C_\sigma\cdot \pa{\sum_{i=1}^{i_0}i\cdot\bar s_i}}
            {m\cdot\sum_{i=1}^{i_0}i\cdot\bar s_i}
          \ = \ C_{2m-2}\cdot (2m-1)\cdot C_\sigma.
      \end{align}
      By the inductive hypothesis, we conclude $C_{2m} = (2m-1)!!\cdot C_\sigma^m$.
      By our technical assumption, $C_\sigma\neq 0$, so $C_{2m}\neq 0$ and thus the degree of $Q_{2m}$ is exactly $m$.

      We can perform the same computation to show that
      the $\tilde\mu_n(2m+1)$ can be expressed as
      the sum of an $m$\textsuperscript{th} degree polynomial in $n$
      and an exponentially small term.
      To see this, define a sequence of polynomials $\{R_n(x)\}$ via
      \begin{equation}
        \label{eq:tilde-beta-def-2}
        R_n(x)
          \ \defeq\  \sum_{\ell=1}^{2m+1}\sum_{i=1}^{i_0}\sum_{j=0}^{j_0}
            \frac{\Omega_{n-i}t_{i,j}}{\Omega_n}
            \cdot(j+\mu_{n-i}-\mu_n)^\ell
            \cdot Q_{2m+1-\ell}(x-i).
      \end{equation}
      Just as for the $2m$\textsuperscript{th} moments, set
      \begin{equation}
        s_{i,n}\ = \ \frac{\Omega_{n-i}}{\Omega_n}\sum_{j=0}^{j_0}t_{i,j}.
      \end{equation}
      Then
      \begin{align}
        \tilde\mu_n(2m+1)
          \ = \  \sum_{i=1}^{i_0}s_{i,n}\tilde\mu_{n-i}(2m+1) + R_n(n).
      \end{align}
      Note that $R_n(x)$ is the sum of finitely many
      polynomials that,
      by Lemma \ref{fact:exponential-convergence-arithmetic},
      converge exponentially quickly.
      Thus $R_n(x)$ converges exponentially quickly to
      \begin{equation}
        \label{eq:bar-R-definition-2}
        \bar R(x)
          \ \defeq\  \sum_{\ell=1}^{2m+1}\sum_{i=1}^{i_0}\sum_{j=0}^{j_0}
            \frac{t_{i,j}}{\lambda_1^i}
            \cdot(j-C_\mu i)^\ell
            \cdot Q_{2m+1-\ell}(x-i).
      \end{equation}
      Furthermore, we have
      $\deg \bar R(x)\le m-1$.
      Indeed, $Q_{2m}$ has degree $m$, so to show that
      $\deg \bar R(x)\le m-1$, we simply need to show
      that the coefficient of $x^m$ is 0.
      Indeed, looking at the $x^m$ coefficients of \eqref{eq:bar-R-definition-2} gives
      \begin{align}
        \label{eq:tilde-beta-leading-coefficient-2}
        [x^m](\bar R(x))
          &\ = \  \sum_{\ell=1}^{2m}\sum_{i=1}^{i_0}\sum_{j=0}^{j_0}
            \binom{2m}{\ell}
            \frac{t_{i,j}}{\lambda_1^i}
            \cdot(j-C_\mu i)^\ell
            \cdot [x^{m}](Q_{2m+1-\ell}(x-i)) \nonumber\\
          &\ = \ C_{2m}\cdot\sum_{i=1}^{i_0}\sum_{j=0}^{j_0}
            2m\cdot
            \frac{t_{i,j}}{\lambda_1^i}
            \cdot(j-C_\mu i)^1
            \ = \  C_{2m}\cdot 2m\cdot 0\ = \ 0.
      \end{align}
      The second to last equality follows from the definition of $C_\mu$ in \ref{eq:C_mu-and-C_sigma}.
      Again, applying Lemma \ref{lem:key-technical-polyn} gives that there exists a degree $\deg\bar R(x) + 1$ polynomial $Q_{2m+1}(x)$ such that $\tilde\mu_n(2m+1)=Q_{2m+1}(n)+O(\gamma_{2m+1}^n)$.
      Since $\deg\bar R(x) + 1\le m$, this completes the induction.
    \end{proof}

  \begin{proof}[Proof of Theorem \ref{thm:2d-recursion}]
    Lemma \ref{lem:mu-calculation} proves the first part of Theorem \ref{thm:2d-recursion}.
    Lemma \ref{lem:moments-calculation} implies that $\sigma_n^2=\tilde\mu_n(2)=Q_2(n)+O(\gamma_2^n)$. Writing $Q_2(n) = C_\sigma n + d_\sigma$ for some $d_\sigma\in\RR$, we have $\sigma_n^2=C_\sigma n + d_\sigma + O(\gamma_2^n)$, proving the second part of Theorem \ref{thm:2d-recursion}.
    We finish the proof of Theorem \ref{thm:2d-recursion}
    by noting that \eqref{eq:moments-gaussian-condition}
    is an immediate consequence of Lemma \ref{lem:moments-calculation}.
  \end{proof}


\section{Gap Theorems}\label{sec:counting-gaps}


\subsection{Gap Recurrence}\label{sec:gap-recurrence}

We start by finding a recurrence relation for an $M \in [G_n, G_{n+1})$ having exactly $k$ gaps of size $g$.
Recall that $k_g(M)$ denotes the number of gaps of size $g$ in the Zeckendorf Decomposition of $M$.

  \begin{lemma}
    \label{lem:zeck-num-gap-count}
    Let $\{G_n\}$ be a positive linear recurrence with recurrence relation
    \begin{equation}
      G_n\ = \ c_1G_{n-1}+\cdots+c_LG_{n-L}
    \end{equation}
    and $c_i>0$ for all $i$.
    Slightly abusing notation (reusing the letter $p$), let
    \begin{align}
      p_{g,n,k}&\ \defeq\  |\{M\in[G_n,G_{n+1}): k_g(M)\ = \ k\}|.
    \end{align}
    Define $d_0=0$ and $d_i=c_1+c_2+\cdots+c_i$ for $1\le i\le L$
    and set $c_i^*=c_i$ for $1\le i < L$ and $c_L^*=c_L - 1$.
    Then there exists $n_0=L+g$ and $k_0 = d_L$ such that,
    for $n\ge n_0, k\ge k_0, g\ge 2$, we have
    \begin{align}
      \label{eq:p_gnk}
      p_{0,n,k}
        &\ = \  \sum_{i=1}^{L}\sum_{j=1}^{c_{i}-1}p_{0,n-i,k-(d_{i-1}-(i-1)+(j-1))}
         + \sum_{i=1}^{L}p_{0,n-i,k-(d_{i-1}-(i-1))}
         \nonumber\\
      p_{1,n,k}
        &\ = \  p_{1,n-1,k} + \sum_{i=1}^{L}(c_{i}-1)p_{1,n-i,k-(i-1)} + \sum_{i=2}^{L}p_{1,n-i,k-(i-2)} \nonumber\\
        &\quad + \sum_{i=1}^{L}(c_{i}-1)\pa{
            \pa{p_{1,n-i,k-i} - p_{1,n-i,k-(i-1)}}
            - \pa{p_{1,n-i-1,k-i} - p_{1,n-i-1,k-(i-1)}}
           } \nonumber\\
      p_{g,n,k}
        &\ = \  \sum_{i=1}^{L}c_{i}p_{g,n-i,k}
          + \sum_{i=1}^{L}c_{i}^*
            \pa{
              \pa{p_{g,n+1-i-g,k-1} - p_{g,n+1-i-g,k}}
              - \pa{p_{g,n-i-g,k-1} - p_{g,n-i-g,k}}
            }.
    \end{align}
  \end{lemma}
  \begin{proof}
    Define
    \begin{align}
      q_{g,n,k}\ \defeq\  |\{M\in[1,G_n): k_g(M)=k\}| \ = \  \sum_{i=1}^{n-1} p_{g,i,k};
    \end{align}
    thus while $p_{g,n,k}$ is the number of $M$ in $[G_n,G_{n+1})$ such that $k_g(M)=k$, $q_{g,n,k}$ is the corresponding quantity for integers in $[1, G_n)$.
    Set $H_{n,0}=0$ and $H_{n,i} = \sum_{i'=1}^i c_{i'}G_{n+1-i'}$ so that, for all $n$, $H_{n,L} = G_{n+1}$.
    Let
    \begin{equation}
      \label{eq:Z-def}
      Z \ = \ \{(i,j)\in\ZZ^2:0\le i\le L-1, 0\le j \le c_{i+1}-1, (i,j)\neq(0,0)\}.
    \end{equation}
    For $n\in\NN$ and $(i,j)\in Z$, let $I_{n,i,j} = [H_{n,i} + jG_{n-i}, H_{n,i} + (j+1)G_{n-i})$ be an interval of integers.
    The $c_1+c_2+\cdots+c_L-1$ intervals $\{I_{n,i,j}:(i,j)\in Z\}$ form a partition of $[G_n,G_{n+1})$, and the sequential order of these intervals is equal to their lexicographical order by $(i,j)$.
    For each $(i,j)\in Z$, we can express $|\{M\in I_{n,i,j}:k_g(M)=k\}$ in terms of $p_{g,n,k}$ and $q_{g,n,k}$ with smaller values of $n$. This is done by case work on whether the smallest term in $H_{n,i}+jG_{n-i}$ (either $G_{n+1-i}$ or $G_{n-i}$ depending on whether $j=0$) is part of a gap of size $g$:
    \begin{align}
      \label{eq:zeck-num-gap-count-intervals}
      |\{M\in I_{n,i,0}:k_0(M)=k\}|
        &\ = \  q_{0,n-i,k-(d_i-i)} \nonumber\\
      |\{M\in I_{n,i,0}:k_1(M)=k\}|
        &\ = \  q_{1,n-i,k-(i-1)} \nonumber\\
      |\{M\in I_{n,i,0}:k_g(M)=k\}|
        &\ = \  q_{g,n-i,k} + p_{g,n+1-i-g,k-1} - p_{g,n+1-i-g,k} \nonumber\\
      |\{M\in I_{n,i,j}:k_0(M)=k\}|
        &\ = \  q_{0,n-i,k-(d_i-i+(j-1))} \nonumber\\
      |\{M\in I_{n,i,j}:k_1(M)=k\}|
        &\ = \  q_{1,n-i,k-i} + p_{g,n-i-1,k-(i+1)} - p_{g,n-i-1,k-i} \nonumber\\
      |\{M\in I_{n,i,j}:k_g(M)=k\}|
        &\ = \  q_{g,n-i,k} + p_{g,n-i-g,k-1}-p_{g,n-i-g,k}
    \end{align}
    (see Appendix \ref{app:zeck-num-gap-count-1} for details).
    These formulas are clean because the number of size-$g$ gaps in an $M=H_{n,i}+jG_{n-i}+M'\in I_{n,i,j}$ is simply the number of size-$g$ gaps in $H_{n,i}+jG_{n-i}$ plus the number of size-$g$ gaps in $M'$ plus possibly one more gap between the two decompositions.
    By definition, for $g\ge0$ we have
    \begin{align}
      p_{g,n,k}
        &\ = \  \sum_{(i,j)\in Z}|\{M\in I_{n,i,j}:k_g(M)=k\}|.
    \end{align}
    From this equation, we can substitute from \eqref{eq:zeck-num-gap-count-intervals},
    plug in the results for $p_{g,n,k}$ and $p_{g,n-1,k}$
    in the expression $p_{g,n,k}-p_{g,n-1,k}$,
    use the identity $q_{g,n,k} - q_{g,n-1,k} = p_{g,n-1,k}$,
    and apply straightforward manipulations
    to obtain the desired result (see Appendix \ref{app:zeck-num-gap-count-2}
    for calculations).
  \end{proof}


\subsection{Proof of Gap Theorems}\label{sec:gap-thms}

Lemma \ref{lem:zeck-num-gap-count} allows us to apply  Theorem \ref{thm:2d-recursion} to the distribution of the number of  fixed sized gaps.
The proof is essentially verifying that the conditions of Theorem \ref{thm:2d-recursion} are met by our gap recurrences.

  \begin{proof}[Proofs of Theorems \ref{thm:gap-lek}, \ref{thm:gap-var}, and
  \ref{thm:gap-gauss}]
    Recall that $k_g(M)$ denotes the number of gaps of size $g$ in the Zeckendorf Decomposition of $M$.
    Let
    \begin{align}
      p_{g,n,k}\ \defeq\  |\{M\in[G_n,G_{n+1}): k_g(M)=k\}|
    \end{align}
    and let $i_0=L+g$, $j_0=d_L$. By Lemma \ref{lem:zeck-num-gap-count}, for every $g\ge 0$, there exist $t_{i,j}$ for $1\le i\le L+g$ and $0\le j\le d_L$ such that for $n>i_0$
    \begin{equation}
      p_{g,n,k}\ = \ \sum_{i=1}^{i_0}\sum_{j=0}^{j_0}t_{i,j}p_{g,n-i,k-j}.
    \end{equation}
    Define $\hat t_i = \sum_{j=0}^{j_0}t_{i,j}$.
    Note that in each recursive formula of \eqref{eq:p_gnk} the terms of the form $p_{g,n-x,y_1} - p_{g,n-x,y_2}$ contribute 0 to $\sum_{j=0}^{j_0}t_{x,j}$, and for each $0\le i\le L-1$ the remaining coefficients of $p_{g,n-i-1,k}$ (over varying $k$) sum to $c_{i+1}$.
    From this we conclude $\hat t_i = c_i$ for $1\le i\le L$ and $\hat t_i = 0$ for $L<i\le i_0$.
    Thus the polynomial
    \be
      T(x)=x^{i_0}-\sum_{i=1}^{i_0}\hat t_ix^{i_0-i} = x^{i_0-L}\pa{x^{L}-\sum_{i=1}^{L}c_ix^{L-i}}
    \ee
    has the maximum root property with some maximum root $\lambda_1>1$
    by Theorem \ref{thm:plrs-maximum-root}.
    Also $\sum_{k=0}^np_{g,n,k} = G_{n+1} - G_n=\Theta(\lambda_1^n)$
    by Theorems \ref{thm:general-binet} and \ref{thm:plrs-maximum-root}.
    As $p_{g,n,k}$ counts something that is well defined when $n\ge 1$ and $k\ge 0$, we have $p_{g,n,k}\ge 0$ for all $n,k$ and $p_{g,n,k}=0$ for $n<0$ or $k<0$.
    Also $p_{g,n,k}=0$ for all $k\ge n$, since no $M\in[G_n,G_{n+1})$ can have a gap greater than $n$.
    Thus there are finitely many pairs $(n,k)$ with $n\le i_0$ such that $p_{g,n,k}\neq 0$.
    Lastly, for every $g$, if the random variable $K_{g,n}$ is nontrivial then the $t_{i,j}$ satisfy
    \begin{equation}
      \label{eq:C_mu/sigma-2}
      C_\mu \ \defeq\  \frac{\sum_{i=1}^{i_0}\sum_{j=0}^{j_0} \frac{t_{i,j}\cdot j}{\lambda_1^i}}{\sum_{i=1}^{i_0}\sum_{j=0}^{j_0} \frac{t_{i,j}\cdot i}{\lambda_1^i}}, \ \ \  C_\sigma \ \defeq\  \frac
            {\sum_{i=1}^{i_0}\sum_{j=0}^{j_0}
            \frac{t_{i,j}}{\lambda_1^i}
            \cdot(j-C_\mu i)^2 }
            {\sum_{i=1}^{i_0}\sum_{j=0}^{j_0}\frac{t_{i,j}\cdot i}{\lambda_1^i}}.
    \end{equation}
    To prove each of $C_\mu>0$ and $C_\sigma>0$, we split into cases on whether $g=0$, $g=1$, or $g\ge 2$.
    For each case we substitute into \eqref{eq:C_mu/sigma-2} and perform standard manipulations (see Appendix \ref{app:C_mu/sigma-positive}).
    Putting these observations together, the proofs follow by applying Theorem \ref{thm:2d-recursion}.
  \end{proof}


\section{Lekkerkerker and Gaussian Summands}
  \label{sec:summands-thms}
  We show the power of Theorem \ref{thm:2d-recursion} by reproving
  Theorems \ref{thm:summands-lek}, \ref{thm:summands-var}, and \ref{thm:summands-gauss}.
  We borrow from the proof given by Miller and Wang \cite{MW1}
  the recursion established for $p_{n,k}$, the number of $M\in[G_n,G_{n+1})$
  with exactly $k$ summands.
  This recursion is extracted as \eqref{eq:num-summands-recursion} from generating functions in \cite{MW1}.
  This recursion can also be found using techniques in the proof of Lemma \ref{lem:zeck-num-gap-count} and the casework for number of summands is simpler than for gaps.
  Miller and Wang's arguments quickly show the mean and variance grow linearly in $n$, but a lot of technical calculations are needed to show the linear coefficients are positive (which is a key ingredient in the proof of the Gaussian behavior). See \cite{CFHMNPX} for another approach, which bypasses the difficulties through an elementary argument involving conditional probabilities, or \cite{B-AM} for a proof through Markov processes.

Similar to \S\ref{sec:gap-thms}, the proof is essentially verifying that the conditions of Theorem \ref{thm:2d-recursion} are met by the summands recursion given by Miller and Wang.

  \begin{proof}[Proofs of Theorems \ref{thm:summands-lek}, \ref{thm:summands-var}, and
  \ref{thm:summands-gauss}]
    Let $p_{n,k}$ be the number of $M\in[G_n,G_{n+1})$ with exactly $k$ summands.
    Then $\Pr[K_{\Sigma,n}=k] = \frac{p_{n,k}}{\sum_{k=0}^\infty p_{n,k}}$.
    Again, $p_{n,k}\ge 0$ for all $n,k$, $p_{n,k}=0$ for all $n<0$ and $k<0$, and $p_{n,k}>0$ for finitely many pairs with $n<L$ as $p_{n,k}=0$ for all $k>n\cdot\max_i(c_i)$, since each $M$ has, for each $a\in\{1,\dots,n\}$, at most $\max_i(c_i)$ copies of $G_a$ in each decomposition.

    Define $d_i=c_1+c_2+\cdots+c_i$ for $1\le i\le L$.
    By Proposition 3.1 from \cite{MW2}, $p_{n,k}$
    satisfies, for $n\ge L$ and $k\ge d_L$,
    \begin{equation}
      \label{eq:num-summands-recursion}
      p_{n,k} \ = \  \sum_{i=1}^L\sum_{j=d_{m-1}}^{d_m-1}p_{n-i,k-j}.
    \end{equation}
    For $1\le i\le L$, $0\le j< d_L$, set $t_{i,j}$ to be $1$ if $d_{i-1}\le j < d_i-1$ and 0 otherwise.
    Defining $\hat t_i\defeq\sum_{j=0}^{d_L-1}t_{i,j}$ gives $\hat t_i = c_i$, and the polynomial
    \be
      T(x)\ =\ x^{i_0}-\sum_{i=1}^{i_0}\hat t_ix^{i_0-i}\ =\ x^{i_0-L}\pa{x^{L}-\sum_{i=1}^{L}c_ix^{L-i}}
    \ee
    has the maximum root property with some maximum root $\lambda_1>1$ by Theorem \ref{thm:plrs-maximum-root}.
    Also $\sum_{k=0}^np_{g,n,k} = G_{n+1} - G_n=\Theta(\lambda_1^n)$
    by Theorems \ref{thm:general-binet} and \ref{thm:plrs-maximum-root}.
    Lastly, since all the $t_{i,j}$ are nonnegative and $t_{n-L, k-(d_L-1)}=1$ with $k-(d_L-1) > 0$, \eqref{eq:C_mu-and-C_sigma} tells us
    \begin{equation}
      C_\mu \ = \  \frac{\sum_{i=1}^{i_0}\sum_{j=0}^{j_0} \frac{t_{i,j}\cdot j}{\lambda_1^i}}
                   {\sum_{i=1}^{i_0}\sum_{j=0}^{j_0} \frac{t_{i,j}\cdot i}{\lambda_1^i}}
            \ \ge \ \frac{\frac{k-(d_L-1)}{\lambda_1^{n-L}}}
                   {\sum_{i=1}^{i_0}\sum_{j=0}^{j_0} \frac{t_{i,j}\cdot i}{\lambda_1^i}}
            \ > \ 0.
    \end{equation}
    Since $t_{1,0} = 1$ and all the $t_{i,j}$ are nonnegative, we have
    \begin{equation}
      C_\sigma
          \ = \  \frac
            {\sum_{i=1}^{i_0}\sum_{j=0}^{j_0}
            \frac{t_{i,j}}{\lambda_1^i}
            \cdot(j-C_\mu i)^2 }
            {\sum_{i=1}^{i_0}\sum_{j=0}^{j_0}\frac{t_{i,j}\cdot i}{\lambda_1^i}}
              \ \ge \  \frac {\frac{t_{1,0}}{\lambda_1^1}
                \cdot(0-C_\mu 1)^2 }
                {\sum_{i=1}^{i_0}\sum_{j=0}^{j_0}\frac{t_{i,j}\cdot i}{\lambda_1^i}}
              \ > \ 0.
    \end{equation}
    Thus we can apply Theorem \ref{thm:2d-recursion}, implying the theorems.
  \end{proof}


\section{Further Work and Open Questions}

We end with a few natural questions for future work.

\begin{enumerate}
  \item Are there other two-dimensional recurrences to which we can apply our central limit type result? The second named author is currently investigating two dimensional sequences and associated notions of legality with colleagues. These lead to recurrence relations,  though the resulting sequences do not have unique decomposition.

  \item Can one remove the constraint that every coefficient $c_i$ must be positive and obtain the same results? Notice that with negative constraints one loses some of the interpretations for the algebra.

  \item What is the rate at which $K_{g,n}$ converges to a normal distribution?
\end{enumerate}


\appendix


\section{$b_n$ is bounded in Lemma \ref{lem:key-technical}}
\label{app:bn-bound}
  Since $s_{i,n}=\bar s_i + O(\gamma^n)$,
  there exist $\alpha_1,\alpha_2 > 0, n_0\in \NN$ such that,
  for all $n\ge n_0+i_0$,
  \begin{align}
    \label{eq:bn-bound-1}
    |b_n|
      &\ = \  \abs{\sum_{i=1}^{i_0}(\bar s_i+O(\gamma^n))b_{n-i} + O(\gamma^n)} \nonumber\\
      &\ \le \  \sum_{i=1}^{i_0}(\bar s_i+\alpha_1\gamma^n)|b_{n-i}| +
      \alpha_2\gamma^n\nonumber\\
      &\ \le \  (1 + i_0\alpha_1\gamma^n)\abs{\max_{1\le i\le i_0}b_{n-i}} +
      \alpha_2\gamma^n\nonumber\\
      &\ \le \  (1 + (i_0\alpha_1+\alpha_2)\gamma^n)
        \abs{\max\pa{1,\max_{1\le i\le i_0}b_{n-i}}} \nonumber\\
      &\ \le \  e^{(i_0\alpha_1+\alpha_2)\gamma^n}
        \abs{\max\pa{1,\max_{1\le i\le i_0}b_{n-i}}}.
  \end{align}
  Let $B=\max_{0\le i< i_0}|b_{n_0+i}|$.
  We prove by induction that
  \begin{equation}
    |b_n|
      \ \le \  (B+1)e^{(i_0\alpha_1+\alpha_2)
          \pa{\gamma^{n_0}+\gamma^{n_0+1}+\cdots+\gamma^n}}.
  \end{equation}
  For $n< n_0+i_0$, we have
  \begin{equation}
    |b_n|
      < (B+1)
      \ \le \  (B+1)e^{(i_0\alpha_1+\alpha_2)
                 (\gamma^{n_0}+\gamma^{n_0+1}+\cdots+\gamma^n)}.
  \end{equation}
  Now assume $n\ge n_0+i_0$, and suppose the assertion is true for $n'< n$.
  Then, by \eqref{eq:bn-bound-1},
  \begin{align}
    |b_n|
      &\ \le \  e^{(i_0\alpha_1+\alpha_2)\gamma^n}
        \abs{\max\pa{1,\max_{1\le i\le i_0}b_{n-i}}} \nonumber \\
      &\ \le \  e^{(i_0\alpha_1+\alpha_2)\gamma^n}
        \abs{\max\pa{1,\max_{1\le i\le i_0}
                     (B+1)e^{(i_0\alpha_1+\alpha_2)
                     (\gamma^{n_0}+\gamma^{n_0+1}+\cdots+\gamma^{n-i})}
                    }} \nonumber \\
      &\ = \  e^{(i_0\alpha_1+\alpha_2)\gamma^n}
        (B+1)e^{(i_0\alpha_1+\alpha_2)
                     (\gamma^{n_0}+\gamma^{n_0+1}+\cdots+\gamma^{n-1})}
                     \nonumber \\
      &\ \le \  (B+1)e^{(i_0\alpha_1+\alpha_2)
                     (\gamma^{n_0}+\gamma^{n_0+1}+\cdots+\gamma^n)}
  \end{align}
  completing the induction. Thus we have
  \begin{equation}
    |b_n|\ \le \  (B+1)e^{(i_0\alpha_1+\alpha_2)\cdot\frac{\gamma^{n_0}}{1-\gamma}},
  \end{equation}
  so the sequence $\{b_n\}$ is bounded.


\section{Decomposition of $b_n$ into similar sequences \ref{lem:key-technical}}
\label{app:bn-decomposition}

We prove by induction that
\begin{equation}
  b_n \ = \  b\id{\init}_n+\sum_{m=i_0}^\infty b\id{m}_n.
\end{equation}
For $n < i_0$, we have
\begin{equation}
  b_n
  \ = \  b\id{\init}_n
  \ = \  b\id{\init}_n+\sum_{m=i_0}^\infty b\id{m}_n.
\end{equation}
Suppose the statement is true for $n'<n$.
Using the fact that $f(n)=b_n\id{n}$
and the recursive definitions for
$b_n\id{\init}$ and $b_n\id{m}$, we obtain
\begin{align}
  b_n
    &\ = \  f(n) + \pa{\sum_{i=1}^{i_0}\bar s_i b_{n-i}} \nonumber \\
    &\ = \  f(n) + \sum_{i=1}^{i_0}\bar s_i
        \pa{b_{n-i}\id{\init}
            + \sum_{m=i_0}^\infty b_{n-i}\id{m}}
       \nonumber \\
    &\ = \  f(n)
     + \sum_{i=1}^{i_0}\bar s_i b_{n-i}\id{\init}
     + \sum_{i=1}^{i_0}\sum_{m=i_0}^\infty \bar s_i b_{n-i}\id{m}
     \nonumber \\
    &\ = \  f(n)
     + \sum_{i=1}^{i_0}\bar s_i b_{n-i}\id{\init}
     + \sum_{m=i_0}^{n-1}\sum_{i=1}^{i_0}\bar s_i b_{n-i}\id{m}
     + \sum_{m=n}^{\infty}\sum_{i=1}^{i_0}\bar s_i b_{n-i}\id{m}
     \nonumber \\
    &\ = \  b_n\id{n} + b_n\id{\init} + \sum_{m=i_0}^{n-1} b_n\id{m} +
      \sum_{m=n}^\infty\sum_{i=1}^{i_0}\bar s_i\cdot 0
      \nonumber \\
    &\ = \  b_n\id{\init} + \sum_{m=i_0}^{n} b_n\id{m}
      \nonumber \\
    &\ = \  b_n\id{\init} + \sum_{m=i_0}^{n} b_n\id{m} +
      \sum_{m=n+1}^\infty b_n\id{m}
      \nonumber \\
    &\ = \  b_n\id{\init} + \sum_{m=i_0}^{\infty} b_n\id{m},
\end{align}
completing the induction.


\section{Computing $b_n$ in Lemma \ref{lem:key-technical-polyn}}
\label{app:bn-simplify}
Our goal this section is to simplify \eqref{eq:key-technical-polyn-an} using the substitution $a_n - b_n = C\cdot n^{D+1}$ where $C=\frac{\bar r_D}{(D+1)\sum_{i=1}^{i_0} i\cdot \bar s_i}$.
We compute a recursive formula for $b_n$ to obtain a linear combination of smaller $b_{n-i}$s plus a polynomial in $n$. For any choice of $C$, the resulting coefficient of $n^{D+1}$ in this polynomial is 0, but for our specific choice of $C$, the $n^D$ term also disappears, so the remaining polynomial has degree at most $D-1$.

Let $\gamma_0\in(\max(\gamma_r,\gamma_s),1)$.
Since $a_n = \sum_{i=1}^{i_0}s_{i,n}a_{n-i} + \sum_{j=0}^Dr_{j,n}n^j$,
we can write
\begin{align}
  b_n
    \ = \ &\,\, a_n - C\cdot n^{D+1}
      \nonumber\\
    \ = \ &\,\, a_n - \sum_{i=1}^{i_0} C s_{i,n} n^{D+1}
      \nonumber\\
    \ = \ & \sum_{i=1}^{i_0}s_{i,n}a_{n-i} + \sum_{j=0}^Dr_{j,n}n^j.
\end{align}
Substituting $a_n - b_n = C\cdot n^{D+1}$ gives
\begin{align}
  b_n
    \ = \ & \pa{\sum_{i=1}^{i_0}s_{i,n}\pa{b_{n-i}+C(n-i)^{D+1}}}
     + \sum_{j=0}^Dr_{j,n}n^j
     - \sum_{i=1}^{i_0} C s_{i,n} n^{D+1}.
\end{align}
We now expand $(n-i)^{D+1}$ to get
\begin{align}
  b_n
    \ = \ & \pa{\sum_{i=1}^{i_0}s_{i,n}\pa{b_{n-i}+C
          \cdot\sum_{j=0}^{D+1}(-1)^{D+1-j}\binom{D+1}{j}n^ji^{D+1-j}}}
           + \sum_{j=0}^Dr_{j,n}n^j
           - \sum_{i=1}^{i_0} C s_{i,n} n^{D+1} \nonumber\\
    =& \sum_{i=1}^{i_0}s_{i,n}b_{n-i}
       + \pa{\sum_{j=0}^{D+1}\sum_{i=1}^{i_0}
        C s_{i,n}(-1)^{D+1-j}\binom{D+1}{j}n^ji^{D+1-j}}
          + \sum_{j=0}^Dr_{j,n}n^j
          - \sum_{i=1}^{i_0} C s_{i,n} n^{D+1}.
\end{align}
As $\sum_{i=1}^{i_0}Cs_{i,n}$ is the coefficient of
$n^{D+1}$ in the binomial expansion, we can cancel to get
\begin{align}
  b_n
    \ = \ & \sum_{i=1}^{i_0}s_{i,n}b_{n-i}
       + \pa{\sum_{j=0}^{D}\sum_{i=1}^{i_0}
        C s_{i,n}(-1)^{D+1-j}\binom{D+1}{j}n^ji^{D+1-j}}
          + \sum_{j=0}^Dr_{j,n}n^j.
\end{align}
We can also pull out the $n^D$ terms of the binomial expansions to get
\begin{align}
  b_n
    \ = \ & \sum_{i=1}^{i_0}s_{i,n}b_{n-i}
       + \pa{\sum_{j=0}^{D-1}\sum_{i=1}^{i_0}
        C s_{i,n}(-1)^{D+1-j}\binom{D+1}{j}n^ji^{D+1-j}}
          \nonumber\\
          &+ \sum_{j=0}^Dr_{j,n}n^j
          - \sum_{i=1}^{i_0}C s_{i,n} (D+1) n^D i
          \nonumber\\
    \ = \ & \sum_{i=1}^{i_0}s_{i,n}b_{n-i}
       + \pa{\sum_{j=0}^{D-1}\sum_{i=1}^{i_0}
        C s_{i,n}(-1)^{D+1-j}\binom{D+1}{j}n^ji^{D+1-j}}
          \nonumber\\
          &+ \sum_{j=0}^{D-1}r_{j,n}n^j
            + r_{D,n}n^D
            - C(D+1)n^D \sum_{i=1}^{i_0}s_{i,n}i.
\end{align}
Now we substitute the value of $C$ in. Note
that $C$ is chosen so that
the coefficient of $n^D$ becomes $O(\gamma_0^n)$.
This happens as $\frac{\sum_{i=1}^{i_0}i\cdot s_{i,n}}{\sum_{i=1}^{i_0} i\cdot\bar s_i}$
and $\frac{r_{D,n}}{\bar r_D}$ are of the form
$1+O(\gamma_s^n)$ and $1+O(\gamma_r^n)$ respectively. We find
\begin{align}
  b_n
    \ = \ & \sum_{i=1}^{i_0}s_{i,n}b_{n-i}
       + \pa{\sum_{j=0}^{D-1}\sum_{i=1}^{i_0}
       C s_{i,n}(-1)^{D+1-j}\binom{D+1}{j}n^ji^{D+1-j}}
       \nonumber\\
         &+ \sum_{j=0}^{D-1}r_{j,n}n^j
             + r_{D,n}n^D
              - \frac{\bar r_D\cdot n^D}{\sum_{i=1}^{i_0} i\cdot\bar s_i}
                \sum_{i=1}^{i_0}s_{i,n}\cdot i
          \nonumber\\
    \ = \ & \sum_{i=1}^{i_0}s_{i,n}b_{n-i}
       + \pa{\sum_{j=0}^{D-1}\sum_{i=1}^{i_0}
       C s_{i,n}(-1)^{D+1-j}\binom{D+1}{j}n^ji^{D+1-j}}
       \nonumber\\
         &+ \sum_{j=0}^{D-1}r_{j,n}n^j
            - \bar r_D n^D
              \pa{
                \frac{\sum_{i=1}^{i_0}i\cdot s_{i,n}}{\sum_{i=1}^{i_0} i\cdot\bar s_i}
                - \frac{r_{D,n}}{\bar r_D}
              }
          \nonumber\\
    \ = \ & \sum_{i=1}^{i_0}s_{i,n}b_{n-i}
       + \pa{\sum_{j=0}^{D-1}\sum_{i=1}^{i_0}
       C s_{i,n}(-1)^{D+1-j}\binom{D+1}{j}n^ji^{D+1-j}}
       \nonumber\\
         &+ \sum_{j=0}^{D-1}r_{j,n}n^j
         - \bar r_D n^D\cdot O(\max(\gamma_r,\gamma_s)^n)
          \nonumber\\
    \ = \ & \sum_{i=1}^{i_0}s_{i,n}b_{n-i}
       + \pa{\sum_{j=0}^{D-1}\sum_{i=1}^{i_0}
       C s_{i,n}(-1)^{D+1-j}\binom{D+1}{j}n^ji^{D+1-j}}
       \nonumber\\
         &+ \sum_{j=0}^{D-1}r_{j,n}n^j
         + O(\gamma_0^n)
          \nonumber\\
    \ = \ & \sum_{i=1}^{i_0}s_{i,n}b_{n-i}
       + \sum_{j=0}^{D-1}n^j\cdot
          \ba{\pa{
              \sum_{i=1}^{i_0}C s_{i,n}(-1)^{D+1-j}\binom{D+1}{j}i^{D+1-j}}
              + r_{j,n}}
          + O(\gamma_0^n).
\end{align}


\section{Recursive formula for $\tilde P_{n,m}(x)$ (Equation \eqref{eq:tilde-P_nm})}
  \label{app:tilde-P_nm}

  We wish to prove
  \begin{equation}
    \tilde P_{n,m}(x)
      \ = \  \sum_{i=1}^{i_0}\sum_{j=0}^{j_0}t_{i,j}
        \sum_{\ell=0}^m
          \binom{m}{\ell}(j+\mu_{n-i}-\mu_n)^\ell
            \tilde P_{n-i,m-\ell}(x)\cdot x^{j+\mu_{n-i}-\mu_n}.
  \end{equation}
  The base case $m=0$ is given by \eqref{eq:tilde-P_n0}.
  Now let $m\ge 1$, and suppose
  \begin{align}
    \tilde P_{n,m-1}(x)
      \ = \  \sum_{i=1}^{i_0}\sum_{j=0}^{j_0}t_{i,j}
        \sum_{\ell=0}^{m-1}
          \binom{m-1}{\ell}(j+\mu_{n-i}-\mu_n)^\ell
            \tilde P_{n-i,m-1-\ell}(x)\cdot x^{j+\mu_{n-i}-\mu_n}.
  \end{align}
  Then
  \begin{align}
    \tilde P_{n,m}(x)
      &\ = \  (x\tilde P_{n,m-1}(x))' \nonumber\\
      &\ = \  \pa{
        x\sum_{i=1}^{i_0}\sum_{j=0}^{j_0}t_{i,j}
        \sum_{\ell=0}^{m-1}
          \binom{m-1}{\ell}(j+\mu_{n-i}-\mu_n)^\ell
            \tilde P_{n-i,m-1-\ell}(x)\cdot x^{j+\mu_{n-i}-\mu_n}
      }' \nonumber\\
      &\ = \ \sum_{i=1}^{i_0}\sum_{j=0}^{j_0}t_{i,j}
        \sum_{\ell=0}^{m-1}
          \binom{m-1}{\ell}(j+\mu_{n-i}-\mu_n)^\ell
            \pa{x\tilde P_{n-i,m-1-\ell}(x)\cdot x^{j+\mu_{n-i}-\mu_n}}'
        \nonumber\\
      &\ = \ \sum_{i=1}^{i_0}\sum_{j=0}^{j_0}t_{i,j}
        \sum_{\ell=0}^{m-1}
          \binom{m-1}{\ell}(j+\mu_{n-i}-\mu_n)^\ell
            \bigg[
              \pa{\tilde P_{n-i,m-\ell}(x)\cdot x^{j+\mu_{n-i}-\mu_n}}\nonumber\\
              &\quad+ \pa{x\tilde P_{n-i,m-1-\ell}(x)\cdot (j+\mu_{n-i}-\mu_n)x^{j+\mu_{n-i}-\mu_n-1}}
            \bigg]\nonumber\\
      &\ = \ \sum_{i=1}^{i_0}\sum_{j=0}^{j_0}t_{i,j}
        \sum_{\ell=0}^{m-1}
          \binom{m-1}{\ell}(j+\mu_{n-i}-\mu_n)^\ell
            \tilde P_{n-i,m-\ell}(x) x^{j+\mu_{n-i}-\mu_n}
            \nonumber\\
        &\quad+\sum_{i=1}^{i_0}\sum_{j=0}^{j_0}t_{i,j}
        \sum_{\ell=0}^{m-1}
          \binom{m-1}{\ell}(j+\mu_{n-i}-\mu_n)^{\ell+1}
            \tilde P_{n-i,m-1-\ell}(x) x^{j+\mu_{n-i}-\mu_n}
            \nonumber\\
      &\ = \ \sum_{i=1}^{i_0}\sum_{j=0}^{j_0}t_{i,j}
        \sum_{\ell=0}^{m-1}
          \binom{m-1}{\ell}(j+\mu_{n-i}-\mu_n)^\ell
            \tilde P_{n-i,m-\ell}(x) x^{j+\mu_{n-i}-\mu_n}
            \nonumber\\
        &\quad+\sum_{i=1}^{i_0}\sum_{j=0}^{j_0}t_{i,j}
        \sum_{\ell=1}^{m}
          \binom{m-1}{\ell-1}(j+\mu_{n-i}-\mu_n)^{\ell}
            \tilde P_{n-i,m-\ell}(x) x^{j+\mu_{n-i}-\mu_n}.
  \end{align}
  Taking out the $\ell=0$ term from the first sum
  and the $\ell=m$ term from the latter one,
  and pairing the remaining terms by common $\ell$,
  we obtain
  \begin{align}
    \tilde P_{n,m}(x)
      &\ = \ \sum_{i=1}^{i_0}\sum_{j=0}^{j_0}t_{i,j}
        \sum_{\ell=1}^{m-1}
          \pa{\binom{m-1}{\ell-1}+\binom{m-1}{\ell}}(j+\mu_{n-i}-\mu_n)^\ell
            \tilde P_{n-i,m-\ell}(x) x^{j+\mu_{n-i}-\mu_n}
            \nonumber\\
        &\quad+\sum_{i=1}^{i_0}\sum_{j=0}^{j_0}t_{i,j}
          \binom{m-1}{0}(j+\mu_{n-i}-\mu_n)^{0}
            \tilde P_{n-i,m-0}(x) x^{j+\mu_{n-i}-\mu_n}
            \nonumber\\
        &\quad+\sum_{i=1}^{i_0}\sum_{j=0}^{j_0}t_{i,j}
          \binom{m-1}{m-1}(j+\mu_{n-i}-\mu_n)^{m}
            \tilde P_{n-i,m-m}(x) x^{j+\mu_{n-i}-\mu_n}
            \nonumber\\
      &\ = \ \sum_{i=1}^{i_0}\sum_{j=0}^{j_0}t_{i,j}
        \sum_{\ell=1}^{m-1}
          \binom{m}{\ell}(j+\mu_{n-i}-\mu_n)^\ell
            \tilde P_{n-i,m-\ell}(x) x^{j+\mu_{n-i}-\mu_n}
            \nonumber\\
        &\quad+\sum_{i=1}^{i_0}\sum_{j=0}^{j_0}t_{i,j}
          \binom{m}{0}(j+\mu_{n-i}-\mu_n)^{0}
            \tilde P_{n-i,m-0}(x) x^{j+\mu_{n-i}-\mu_n}
            \nonumber\\
        &\quad+\sum_{i=1}^{i_0}\sum_{j=0}^{j_0}t_{i,j}
          \binom{m}{m}(j+\mu_{n-i}-\mu_n)^{m}
            \tilde P_{n-i,m-m}(x) x^{j+\mu_{n-i}-\mu_n}
            \nonumber\\
      &\ = \ \sum_{i=1}^{i_0}\sum_{j=0}^{j_0}t_{i,j}
        \sum_{\ell=0}^{m}
          \binom{m}{\ell}(j+\mu_{n-i}-\mu_n)^\ell
            \tilde P_{n-i,m-\ell}(x) x^{j+\mu_{n-i}-\mu_n}
  \end{align}
  as desired. This completes the induction.


\section{Details for recursively counting gaps in Lemma \ref{lem:zeck-num-gap-count}}

  \subsection{Computing $|\{M\in[H_{n,i} + jG_{n-i}, H_{n,i}+(j+1)G_{n-i}):k_g(M)=k\}|$}
    \label{app:zeck-num-gap-count-1}
    In this section, we compute a recursive formula for $|\{M\in I_{n,i,j}: k_g(M)=k\}|$. Note that the number of size-$g$ gaps in an $M=H_{n,i}+jG_{n-i}+M'\in I_{n,i,j}$ is simply the number of size-$g$ gaps in $H_{n,i}+jG_{n-i}$ plus the number of size-$g$ gaps in $M'$ plus possibly one more gap between the two decompositions. This observation gives a clean way to produce the desired recursions. We need to be careful in our case work as the number of size-$g$ gaps in $H_{n,i}+jG_{n-i}$ varies depending on whether $g$ is 0, 1, or at least 2, and the the existence of the gap between the smallest term in the decomposition of $H_{n,i}+jG_{n-i}$ and the largest term in the decomposition of $M'$ depends on whether $j$ is nonzero.

    For $j\ge 0$, the decomposition of any element of $I_{n,i,j} = [H_{n,i} + jG_{n-i}, H_{n,i}+(j+1)G_{n-i})$ begins with the decomposition of $H_{n,i}+jG_{n-i}$, and the decomposition of $H_{n,i}+jG_{n-i}$ contains only gaps of sizes 0 and 1 by definition of $H_{n,i}$.
    In particular, when $j=0$, the decomposition of $H_{n,i}$
    contains $d_i-i$ gaps of size 0 and $i-1$ gaps of size 1.
    When $j\ge 1$, the decomposition of $H_{n,i}+jG_{n-i}$
    contains $d_i-i + (j-1)$ gaps of size 0 and $i$ gaps of size 1.
    For $i\ge 1$, this gives
    \begin{align}
      \label{eq:j=0-interval-step-1}
      &|\{M\in[H_{n,i}, H_{n,i}+G_{n-i}):k_0(M)=k\}|\nonumber\\
        &\quad\ = \  |\{M\in[G_{n+1-i},G_{n+1-i}+G_{n-i}):k_0(M)=k-(d_i-i)\}| \nonumber\\
      &|\{M\in[H_{n,i}, H_{n,i}+G_{n-i}):k_1(M)=k\}|\nonumber\\
        &\quad\ = \  |\{M\in[G_{n+1-i},G_{n+1-i}+G_{n-i}):k_1(M)=k-(i-1)\}| \nonumber\\
      &|\{M\in[H_{n,i}, H_{n,i}+G_{n-i}):k_g(M)=k\}|\nonumber\\
        &\quad\ = \  |\{M\in[G_{n+1-i},G_{n+1-i}+G_{n-i}):k_g(M)=k\}|,
    \end{align}
    and for $i\ge 0, j\ge 1$, we have
    \begin{align}
      \label{eq:j-ge-1-interval-step-1}
      &|\{M\in[H_{n,i} + jG_{n-i}, H_{n,i}+(j+1)G_{n-i}):k_0(M)=k\}|\nonumber\\
        &\quad\ = \  |\{M\in[G_{n-i},2G_{n-i}):k_0(M)=k-(d_i - i + (j-1))\}| \nonumber\\
      &|\{M\in[H_{n,i} + jG_{n-i}, H_{n,i}+(j+1)G_{n-i}):k_1(M)=k\}|\nonumber\\
        &\quad\ = \  |\{M\in[G_{n-i},2G_{n-i}):k_1(M)=k-i\}| \nonumber\\
      &|\{M\in[H_{n,i} + jG_{n-i}, H_{n,i}+(j+1)G_{n-i}):k_g(M)=k\}|\nonumber\\
        &\quad\ = \  |\{M\in[G_{n-i},2G_{n-i}):k_g(M)=k\}|.
    \end{align}
    We can further push \eqref{eq:j=0-interval-step-1}
    by noting that, for $M\in[G_{n+1-i},G_{n+1-i}+G_{n-i})$,
    the decomposition of $M$ begins with $G_{n+1-i}$, and
    furthermore $G_{n+1-i}$ is not a part of a gap of size 0 or 1.
    Thus
    \begin{align}
      \label{eq:j=0-interval-step-2}
      |\{M\in\,&I_{n,i,0}:k_0(M)=k\}| \nonumber\\
        &\ = \  |\{M\in[H_{n,i}, H_{n,i}+G_{n-i}):k_0(M)=k\}| \nonumber\\
        &\ = \  |\{M\in[G_{n+1-i},G_{n+1-i}+G_{n-i}):k_0(M)=k-(d_i-i)\}| \nonumber\\
        &\ = \  |\{M\in[0,G_{n-i}):k_0(M)=k-(d_i-i)\}| \nonumber\\
        &\ = \  q_{0,n-i,k-(d_i-i)} \nonumber\\
      |\{M\in\,&I_{n,i,0}:k_1(M)=k\}| \nonumber\\
        &\ = \  |\{M\in[G_{n+1-i},G_{n+1-i}+G_{n-i}):k_1(M)=k-(i-1)\}| \nonumber\\
        &\ = \  |\{M\in[0,G_{n-i}):k_1(M)=k-(i-1)\}| \nonumber\\
        &\ = \  q_{1,n-i,k-(i-1)}.
    \end{align}
    Similarly, the decomposition of any $M\in[G_{n-i},2G_{n-i})$ begins with $G_{n-i}$,
    and $G_{n-i}$ is not a part of a gap of size 0, so we have
    \begin{align}
      \label{eq:j-ge-1-interval-step-2}
      |\{M\in\,&I_{n,i,j}:k_0(M)=k\}| \nonumber\\
        &\ = \ |\{M\in[H_{n,i} + jG_{n-i}, H_{n,i}+(j+1)G_{n-i}):k_0(M)=k\}| \nonumber\\
        &\ = \  |\{M\in[G_{n-i},2G_{n-i}):k_0(M)=k-(d_i - i + (j-1))\}| \nonumber\\
        &\ = \  |\{M\in[0,G_{n-i}):k_0(M)=k-(d_i - i + (j-1))\}| \nonumber\\
        &\ = \  q_{0,n-i,k-(d_i-i+(j-1))}.
    \end{align}
    For $g\ge2$ we have
    \begin{align}
      |\{M\in\,&I_{n,i,0}:k_g(M)=k\}| \nonumber\\
        &\ = \  |\{M\in[H_{n,i}, H_{n,i}+G_{n-i}):k_g(M)=k\}| \nonumber\\
        &\ = \  |\{M\in[G_{n+1-i},G_{n+1-i}+G_{n-i}):k_g(M)=k\}| \nonumber\\
        &\ = \  |\{M\in[G_{n+1-i},G_{n+1-i}+G_{n+1-i-g}):k_g(M)=k\}| \nonumber\\
        &\quad + |\{M\in[G_{n+1-i}+G_{n+1-i-g},G_{n+1-i}+G_{n+2-i-g}):k_g(M)=k\}|
          \nonumber\\
        &\quad + |\{M\in[G_{n+1-i}+G_{n+2-i-g},G_{n+1-i}+G_{n-i}):k_g(M)=k\}|
          \nonumber\\
        &\ = \  |\{M\in[0,G_{n+1-i-g}):k_g(M)=k\}| \nonumber\\
        &\quad + |\{M\in[G_{n+1-i-g},G_{n+2-i-g}):k_g(M)=k-1\}| \nonumber\\
        &\quad + |\{M\in[G_{n+2-i-g},G_{n-i}):k_g(M)=k\}| \nonumber\\
        &\ = \  |\{M\in[0,G_{n-i}):k_g(M)=k\}| \nonumber\\
        &\quad + |\{M\in[G_{n+1-i-g},G_{n+2-i-g}):k_g(M)=k-1\}| \nonumber\\
        &\quad - |\{M\in[G_{n+1-i-g},G_{n+2-i-g}):k_g(M)=k\}| \nonumber\\
        &\ = \  q_{g,n-i,k} + p_{g,n+1-i-g,k-1} - p_{g,n+1-i-g,k},
    \end{align}
    where the third equality comes from noting that for $M\in[G_{n+1-i}, G_{n+1-i}+G_{n-i})$,
    $G_{n+1-i}$ is part of a gap of size $g$ in the decomposition of $M$ if and only if
    $M\in[G_{n+1-i}+G_{n+1-i-g},G_{n+1-i}+G_{n+2-i-g})$.
    Using the same argument, we obtain, for $g\ge 1$ and $j\ge 1$,
    \begin{align}
      |\{M\in\,&[G_{n-i},2G_{n-i}):k_g(M)=k\}| \nonumber\\
        &\ = \  |\{M\in[G_{n-i},G_{n-i}+G_{n-i-g}):k_g(M)=k\}| \nonumber\\
        &\quad + |\{M\in[G_{n-i}+G_{n-i-g},G_{n-i}+G_{n+1-i-g}):k_g(M)=k\}|
          \nonumber\\
        &\quad + |\{M\in[G_{n-i}+G_{n+1-i-g},2G_{n-i}):k_g(M)=k\}|
          \nonumber\\
        &\ = \  |\{M\in[0,G_{n-i-g}):k_g(M)=k\}| \nonumber\\
        &\quad + |\{M\in[G_{n-i-g},G_{n+1-i-g}):k_g(M)=k-1\}| \nonumber\\
        &\quad + |\{M\in[G_{n+1-i-g},G_{n-i}):k_g(M)=k\}| \nonumber\\
        &\ = \  |\{M\in[0,G_{n-i}):k_g(M)=k\}| \nonumber\\
        &\quad + |\{M\in[G_{n-i-g},G_{n+1-i-g}):k_g(M)=k-1\}| \nonumber\\
        &\quad - |\{M\in[G_{n-i-g},G_{n+1-i-g}):k_g(M)=k\}| \nonumber\\
        &\ = \  q_{g,n-i,k} + p_{g,n-i-g,k-1} - p_{g,n-i-g,k}.
    \end{align}
    Combining with \eqref{eq:j-ge-1-interval-step-1},
    we have, for $g\ge 2$,
    \begin{align}
      |\{M\in\,&I_{n,i,j}:k_1(M)=k\}| \nonumber\\
        &\ = \  |\{M\in[H_{n,i} + jG_{n-i}, H_{n,i}+(j+1)G_{n-i}):k_1(M)=k\}| \nonumber\\
        &\ = \  |\{M\in[G_{n-i},2G_{n-i}):k_1(M)=k-i\}| \nonumber\\
        &\ = \  q_{1,n-i,k-i} + p_{1,n-i-1,k-i-1} - p_{1,n-i-1,k-i} \nonumber\\
      |\{M\in\,&I_{n,i,j}:k_g(M)=k\}|\nonumber\\
        &\ = \  |\{M\in[H_{n,i} + jG_{n-i}, H_{n,i}+(j+1)G_{n-i}):k_g(M)=k\}|\nonumber\\
        &\ = \  |\{M\in[G_{n-i},2G_{n-i}):k_g(M)=k\}| \nonumber\\
        &\ = \  q_{g,n-i,k} + p_{g,n-i-g,k-1} - p_{g,n-i-g,k}.
    \end{align}
    This establishes all six equalities that we desire.


  \subsection{Computing $p_{g,n,k}$}
    \label{app:zeck-num-gap-count-2}
    This section uses careful bookkeeping to produce homogenous two dimensional recursive formulas for $p_{g,n,k}$ using \eqref{eq:zeck-num-gap-count-intervals}.

    Recall that for $g\ge0$ we have
    \begin{align}
      p_{g,n,k}
        &\ = \  \sum_{(i,j)\in Z}|\{M\in I_{n,i,j}:k_g(M)=k\}| \nonumber\\
        &\ = \  \sum_{i=0}^{L-1}\sum_{j=1}^{c_{i+1}-1}
            |\{M\in[H_{n,i} + jG_{n-i}, H_{n,i}+(j+1)G_{n-i}):k_g(M)=k\}|
            \nonumber\\
        &\quad + \sum_{i=1}^{L-1}
            |\{M\in[H_{n,i}, H_{n,i}+G_{n-i}):k_g(M)=k\}|.
    \end{align}
    Substituting from \eqref{eq:zeck-num-gap-count-intervals}, we have (for $g\ge 2$)
    \begin{align}
      p_{0,n,k}
        &\ = \  \sum_{i=0}^{L-1}\sum_{j=1}^{c_{i+1}-1}q_{0,n-i,k-(d_i-i+(j-1))}
         + \sum_{i=1}^{L-1}q_{0,n-i,k-(d_i-i)}
         \nonumber\\
      p_{1,n,k}
        &\ = \  \sum_{i=0}^{L-1}\sum_{j=1}^{c_{i+1}-1}
              \pa{q_{1,n-i,k-i} + p_{1,n-i-1,k-i-1} - p_{1,n-i-1,k-i}}
         + \sum_{i=1}^{L-1} q_{1,n-i,k-(i-1)}
         \nonumber\\
        &\ = \  \sum_{i=0}^{L-1}
            (c_{i+1}-1)
            \pa{q_{1,n-i,k-i} + p_{1,n-i-1,k-i-1} - p_{1,n-i-1,k-i}}
         + \sum_{i=1}^{L-1} q_{1,n-i,k-(i-1)}
         \nonumber\\
      p_{g,n,k}
        &\ = \  \sum_{i=0}^{L-1}\sum_{j=1}^{c_{i+1}-1}
            \pa{q_{g,n-i,k} + p_{g,n-i-g,k-1} - p_{g,n-i-g,k}} \nonumber\\
         &\quad+ \sum_{i=1}^{L-1}
            \pa{q_{g,n-i,k} + p_{g,n+1-i-g,k-1} - p_{g,n+1-i-g,k}} \nonumber\\
        &\ = \  \sum_{i=0}^{L-1}(c_{i+1}-1)
            \pa{q_{g,n-i,k} + p_{g,n-i-g,k-1} - p_{g,n-i-g,k}} \nonumber\\
         &\quad+ \sum_{i=1}^{L-1}
            \pa{q_{g,n-i,k} + p_{g,n+1-i-g,k-1} - p_{g,n+1-i-g,k}}.
    \end{align}

    Substituting for $p_{g,n,k}$ and $p_{g,n-1,k}$
    and using $q_{g,n,k} - q_{g,n-1,k} = p_{g,n-1,k}$, we obtain for $g=0$
    \begin{align}
      p_{0,n,k} - p_{0,n-1,k}
        &\ = \  \sum_{i=0}^{L-1}\sum_{j=1}^{c_{i+1}-1}p_{0,n-i-1,k-(d_i-i+(j-1))}
         + \sum_{i=1}^{L-1}p_{0,n-i-1,k-(d_i-i)}
         \nonumber\\
      p_{0,n,k}
        &\ = \  \sum_{i=0}^{L-1}\sum_{j=1}^{c_{i+1}-1}p_{0,n-i-1,k-(d_i-i+(j-1))}
         + \sum_{i=0}^{L-1}p_{0,n-i-1,k-(d_i-i)}.
         \nonumber\\
        &\ = \  \sum_{i=1}^{L}\sum_{j=1}^{c_{i}-1}p_{0,n-i,k-(d_{i-1}-(i-1)+(j-1))}
         + \sum_{i=1}^{L}p_{0,n-i,k-(d_{i-1}-(i-1))}.
    \end{align}
    Similarly for $g=1$ we obtain
    \begin{align}
      p_{1,n,k} - p_{1,n-1,k}
        &\ = \  \sum_{i=0}^{L-1}(c_{i+1}-1) \Big[p_{1,n-i-1,k-i}
              + \pa{p_{1,n-i-1,k-i-1} - p_{1,n-i-1,k-i}} \nonumber\\
        &\quad- \pa{p_{1,n-i-2,k-i-1} - p_{1,n-i-2,k-i}}\Big]
              + \sum_{i=1}^{L-1} p_{1,n-i-1,k-(i-1)}. \nonumber\\
    \end{align}
    Thus
    \begin{align}
      p_{1,n,k}
        &\ = \  p_{1,n-1,k} + \sum_{i=0}^{L-1}(c_{i+1}-1)p_{1,n-i-1,k-i} + \sum_{i=1}^{L-1}p_{1,n-i-1,k-(i-1)} \nonumber\\
        &\quad + \sum_{i=0}^{L-1}(c_{i+1}-1)\pa{
            \pa{p_{1,n-i-1,k-i-1} - p_{1,n-i-1,k-i}}
            - \pa{p_{1,n-i-2,k-i-1} - p_{1,n-i-2,k-i}}
           } \nonumber\\
        &\ = \  p_{1,n-1,k} + \sum_{i=1}^{L}(c_{i}-1)p_{1,n-i,k-(i-1)} + \sum_{i=2}^{L}p_{1,n-i,k-(i-2)} \nonumber\\
        &\quad + \sum_{i=1}^{L}(c_{i}-1)\pa{
            \pa{p_{1,n-i,k-i} - p_{1,n-i,k-(i-1)}}
            - \pa{p_{1,n-i-1,k-i} - p_{1,n-i-1,k-(i-1)}}
           },
    \end{align}
    and for $g\ge2$ we have
    \begin{align}
      p_{g,n,k} - p_{g,n-1,k}
        &\ = \  \sum_{i=0}^{L-1}(c_{i+1}-1)
            \Big[p_{g,n-i-1,k} + \pa{p_{g,n-i-g,k-1} - p_{g,n-i-g,k}}\nonumber\\
        &\quad- \pa{p_{g,n-i-g-1,k-1} - p_{g,n-i-g-1,k}} \Big]
          + \sum_{i=1}^{L-1} \Big[p_{g,n-i-1,k} \nonumber\\
        &\quad+ \pa{p_{g,n+1-i-g,k-1} - p_{g,n+1-i-g,k}}
          - \pa{p_{g,n-i-g,k-1} - p_{g,n-i-g,k}}\Big].
         \nonumber\\
    \end{align}
    Thus
    \begin{align}
      p_{g,n,k}
        &\ = \  \sum_{i=0}^{L-1}c_{i+1}p_{g,n-i-1,k} \nonumber\\
        &\quad + \sum_{i=0}^{L-1}(c_{i+1}-1)
            \pa{
              \pa{p_{g,n-i-g,k-1} - p_{g,n-i-g,k}}
              - \pa{p_{g,n-i-g-1,k-1} - p_{g,n-i-g-1,k}}
            }
         \nonumber\\
        &\quad + \sum_{i=1}^{L-1}
            \pa{
              \pa{p_{g,n+1-i-g,k-1} - p_{g,n+1-i-g,k}}
              - \pa{p_{g,n-i-g,k-1} - p_{g,n-i-g,k}}
            }
         \nonumber\\
        &\ = \  \sum_{i=1}^{L}c_{i}p_{g,n-i,k}
          + \sum_{i=1}^{L}c_{i}^*
            \pa{
              \pa{p_{g,n+1-i-g,k-1} - p_{g,n+1-i-g,k}}
              - \pa{p_{g,n-i-g,k-1} - p_{g,n-i-g,k}}
            }
    \end{align}
    as desired.


\section{Proving $C_\mu>0$ and $C_\sigma>0$ in Section \ref{sec:gap-thms}}
  \label{app:C_mu/sigma-positive}
  In this section we prove $C_\mu>0$ and $C_\sigma>0$ in \eqref{eq:C_mu/sigma-2}. We first verify that the denominators of $C_\mu$ and $C_\sigma$ are positive. Because the numerators of $C_\mu$ and $C_\sigma$ are linear in the $t_{i,j}$, we obtain expressions for their numerators directly from the gap recurrences in \eqref{eq:p_gnk}. As the recurrences for the cases $g=0$, $g=1$, and $g\ge 2$ are different, we check that the numerators of $C_\mu$ and $C_\sigma$ are positive for each case separately. Each case is dealt with using standard methods.

  Recall
  \begin{align}
    \label{eq:C_mu/sigma-3}
    C_\mu \ &\defeq\  \frac{\sum_{i=1}^{i_0}\sum_{j=0}^{j_0} \frac{t_{i,j}\cdot j}{\lambda_1^i}}{\sum_{i=1}^{i_0}\sum_{j=0}^{j_0} \frac{t_{i,j}\cdot i}{\lambda_1^i}} \nonumber\\
    C_\sigma \ &\defeq\  \frac
          {\sum_{i=1}^{i_0}\sum_{j=0}^{j_0}
          \frac{t_{i,j}}{\lambda_1^i}
          \cdot(j-C_\mu i)^2 }
          {\sum_{i=1}^{i_0}\sum_{j=0}^{j_0}\frac{t_{i,j}\cdot i}{\lambda_1^i}}.
  \end{align}
  Let
  \begin{align}
    C_\mu^* \ &\defeq \ \sum_{i=1}^{i_0}\sum_{j=0}^{j_0} \frac{t_{i,j}\cdot j}{\lambda_1^i}, \nonumber\\
    C_\sigma^* \ &\defeq \ \sum_{i=1}^{i_0}\sum_{j=0}^{j_0} \frac{t_{i,j}}{\lambda_1^i} \cdot(j-C_\mu i)^2
  \end{align}
  be the numerators of $C_\mu, C_\sigma$ in \eqref{eq:C_mu/sigma-3}, respectively.
  Note that the denomatators of $C_\mu$ and $C_\sigma$ are both always positive as
  \begin{equation}
    \sum_{i=1}^{i_0}\sum_{j=0}^{j_0} \frac{t_{i,j}\cdot i}{\lambda_1^i}
      \ = \ \sum_{i=1}^{i_0} \frac{\hat t_i \cdot i}{\lambda_1^i}
      \ > \ 0
  \end{equation}
  since $\hat t_i = c_i > 0$ for $1\le i \le L$ and $\hat t_i = 0$ for $L < i < i_0$.
  Thus it suffices to prove $C_\mu^* > 0$ and $C_\sigma^* > 0$ when \m{K_{g,n}} is nontrivial.

  We first prove $C_\mu^*>0$. Since $C_\mu^*$ is linear in $t_{i,j}$, \eqref{eq:C_mu/sigma-3} tells us we can obtain $C_\mu^*$ by replacing every instance of $p_{g,n-x,k-y}$ in the recurrence relations of \eqref{eq:p_gnk} with $y/\lambda_1^x$.

  Suppose $g=0$.  If $c_i=1$ for all $1\le i< L$ and $c_L$ is 1 or 2, then no $M$ has gaps of size 0 in the decomposition, so the random variable $K_{g,n}$ is trivial. Otherwise, $c_i\ge 2$ for some $i<L$. In this case, $d_{L-1}-(L-1) > 0$. Thus, evaluating $C_\mu^*$ gives
  \begin{align}
    \label{eq:p_0nk-copy-0}
    p_{0,n,k}
      &\ = \  \sum_{i=1}^{L}\sum_{j=1}^{c_{i}-1}p_{0,n-i,k-(d_{i-1}-(i-1)+(j-1))}
       + \sum_{i=1}^{L}p_{0,n-i,k-(d_{i-1}-(i-1))},
  \end{align}
  so
  \begin{align}
    C_\mu^*
      &\ = \  \sum_{i=1}^{L}\sum_{j=1}^{c_{i}-1}\frac{k-(d_{i-1}-(i-1)+(j-1))}{\lambda_1^i}
       + \sum_{i=1}^{L}\frac{k-(d_{i-1}-(i-1))}{\lambda_1^i} \nonumber\\
      &\ > \  \sum_{i=1}^{L}\sum_{j=1}^{c_{i}-1}\frac{0}{\lambda_1^i}
       + \sum_{i=1}^{L}\frac{0}{\lambda_1^i}\ =\ 0.
  \end{align}

  Now suppose $g=1$. If $c_1=c_2=1$ and $L=2$ (i.e., $\{G_n\}$ is the Fibonaccis), then $K_{1,n}=0$ is trivial, so we can assume otherwise. Recall
  \begin{align}
    \label{eq:p_1nk-copy-0}
    p_{1,n,k}
      &\ = \  p_{1,n-1,k} + \sum_{i=1}^{L}(c_{i}-1)p_{1,n-i,k-(i-1)} + \sum_{i=2}^{L}p_{1,n-i,k-(i-2)} \nonumber\\
      &\quad + \sum_{i=1}^{L}(c_{i}-1)\pa{
          \pa{p_{1,n-i,k-i} - p_{1,n-i,k-(i-1)}}
          - \pa{p_{1,n-i-1,k-i} - p_{1,n-i-1,k-(i-1)}}
         }.
  \end{align}
  Note, when we perform the substitution to obtain $C_\mu^*$, any expression of the form
  \be
    (p_{1,n-(x-1),k-y} - p_{1,n-(x-1),k-(y-1)}) - (p_{1,n-x,k-y} + p_{1,n-x,k-(y-1)})
  \ee
  becomes
  \be
    \frac{y - (y-1)}{\lambda_1^{x-1}} - \frac{y - (y-1)}{\lambda_1^{x}}\ =\ \frac{\lambda_1-1}{\lambda_1^x} \ >\ 0.
  \ee
  Thus \eqref{eq:p_1nk-copy-0} gives that when $g=1$,
  \begin{align}
  C_\mu^*
      &\ = \  \frac{0}{\lambda_1^1} +  \sum_{i=1}^{L}(c_{i}-1)\frac{(i-1)}{\lambda_1^i}
       + \sum_{i=2}^{L}\frac{i-2}{\lambda_1^i}
       + \sum_{i=1}^{L}(c_{i}-1)\frac{\lambda_1-1}{\lambda_1^{i+1}}\ >\ 0.
  \end{align}
  To see that the sum is in fact positive, note first that every summand is nonnegative.
  Furthermore, if any $c_i$ is greater than 1, the last sum is strictly positive.
  Otherwise, all the $c_i$'s are 1, in which case $L\ge 3$ since the sequence is not
  the Fibonaccis. Then the second to last sum is strictly positive.

  Lastly, assume $g\ge 2$. Recall
  \begin{align}
    p_{g,n,k}
      &\ = \  \sum_{i=1}^{L}c_{i}p_{g,n-i,k}
        + \sum_{i=1}^{L}c_{i}^*
          \pa{
            \pa{p_{g,n+1-i-g,k-1} - p_{g,n+1-i-g,k}}
            - \pa{p_{g,n-i-g,k-1} - p_{g,n-i-g,k}}
          }.
  \end{align}
  Performing the same substitution gives
  \begin{align}
    C_\mu^*
      &\ = \  \sum_{i=1}^{L}c_{i}\cdot0
        + \sum_{i=1}^{L}c_{i}^*\cdot\frac{\lambda_1-1}{\lambda_1^{i+g}}\ >\ 0
  \end{align}
  as $c_i^*>0$ for some $i$ by definition. This proves that for any $g$, we have $C_\mu^*>0$, so for any \m{g} we also have \m{C_\mu > 0}.

  We can similarly casework on \m{g} to prove $C_\sigma^*>0$ when $K_{g,n}$ is nontrivial.
  As before $C_\sigma^*$ linear in the $t_{i,j}$, so by \eqref{eq:C_mu/sigma-3} we can obtain $C_\sigma^*$ by replacing every instance of $p_{g,n-x,k-y}$ in \eqref{eq:p_gnk} with $(y-C_\mu x)^2/\lambda_1^x$. This produces an expression for $C_\sigma^*$ that we prove is positive using standard techniques.

  First suppose $g=0$. Recall
  \begin{align}
    \label{eq:p_0nk-copy}
    p_{0,n,k}
      &\ = \  \sum_{i=1}^{L}\sum_{j=1}^{c_{i}-1}p_{0,n-i,k-(d_{i-1}-(i-1)+(j-1))}
       + \sum_{i=1}^{L}p_{0,n-i,k-(d_{i-1}-(i-1))}.
  \end{align}
  By considering $i=1$ and $j=1$ in the double sum of \eqref{eq:p_0nk-copy}, we have $t_{1,0}=1$.
  Since $t_{i,j}\ge 0$ for all $i,j$, we have $C_\sigma^*=0$ if and only if $j-C_\mu i = 0$ for all $i,j$ satisfying $t_{i,j}\neq 0$.
  But this is impossible as $t_{1,0}=1$ implies $C_\mu=0$, and we already showed $C_\mu > 0$ when $K_{g,n}$ is nontrivial.
  Thus $C_\sigma^* > 0$ when $g=0$ and $K_{0,n}$ is nontrivial.

  Now suppose $g=1$. Recall
  \begin{align}
    \label{eq:p_1nk-copy}
    p_{1,n,k}
      &\ = \  p_{1,n-1,k} + \sum_{i=1}^{L}(c_{i}-1)p_{1,n-i,k-(i-1)} + \sum_{i=2}^{L}p_{1,n-i,k-(i-2)} \nonumber\\
      &\quad + \sum_{i=1}^{L}(c_{i}-1)\pa{
          \pa{p_{1,n-i,k-i} - p_{1,n-i,k-(i-1)}}
          - \pa{p_{1,n-i-1,k-i} - p_{1,n-i-1,k-(i-1)}}
         }.
  \end{align}
  Again, if $c_1=c_2=1$ and $L=2$ (i.e., $\{G_n\}$ is the Fibonaccis), then $K_{1,n}=0$ is trivial, so we can assume otherwise.
  Substituting for $C_\sigma^*$ gives
  \begin{align}
    C_\sigma^*
      &\ = \  \frac{(0-C_\mu)^2}{\lambda_1^i} + \sum_{i=1}^{L}(c_{i}-1)\frac{(i-1-C_\mu i)^2}{\lambda_1^i} + \sum_{i=2}^{L}\frac{(i-2-C_\mu i)^2}{\lambda_1^i} \nonumber\\
      &\quad + \sum_{i=1}^{L}(c_{i}-1)\pa{
          \frac{(i-C_\mu i)^2}{\lambda_1^i}
        - \frac{(i-1-C_\mu i)^2}{\lambda_1^i}
        - \frac{(i-C_\mu (i+1))^2}{\lambda_1^{i+1}}
        + \frac{(i-1-C_\mu (i+1))^2}{\lambda_1^{i+1}}
        } \nonumber \\
      &\ = \  \frac{C_\mu^2}{\lambda_1^i} + \sum_{i=2}^{L}\frac{(i-2-C_\mu i)^2}{\lambda_1^i}
        + \sum_{i=1}^{L}(c_{i}-1)\pa{
          \frac{(i-C_\mu i)^2}{\lambda_1^i}
        - \frac{2i-1-2C_\mu (i+1)}{\lambda_1^{i+1}}
        } \nonumber\\
  \end{align}
  By an earlier argument $C_\mu > 0$, so we simplify to get
  \begin{align}
    C_\sigma^*
      &\ > \ \sum_{i=1}^{L}(c_{i}-1)\pa{ \frac{(i-C_\mu i)^2}{\lambda_1^i} - \frac{2i-1-2C_\mu (i+1)}{\lambda_1^{i+1}} }.
  \end{align}
  Now we show that for all $i\ge 1$ we have
  \begin{align}
    (i-C_\mu i)^2 - \frac{2i-1-2C_\mu (i+1)}{\lambda_1}\ \ge\ 0.
  \end{align}
  If $2i-1-2C_\mu(i+1)\ge 0$, then since $\lambda_1 > 1$, we obtain
  \begin{align}
    (i-C_\mu i)^2 &- \frac{2i-1-2C_\mu (i+1)}{\lambda_1} \nonumber\\
    &\ \ge \  (i-C_\mu i)^2 - (2i-1-2C_\mu (i+1)) \nonumber\\
    &\ = \  (i-1-C_\mu i)^2 + 2C_\mu \ > \  0.
  \end{align}
  Otherwise, we have $2i-1-2C_\mu(i+1)<0$ so
  \begin{align}
    (i-C_\mu i)^2 - \frac{2i-1-2C_\mu (i+1)}{\lambda_1} \ > \  (i-C_\mu i)^2 \ \ge \ 0.
  \end{align}
  These two cases allow us to conclude
  \begin{align}
    C_\sigma^*
      &\ > \ \sum_{i=1}^{L}\frac{c_i-1}{\lambda_1^i}\pa{(i-C_\mu i)^2 - \frac{2i-1-2C_\mu (i+1)}{\lambda_1} } \nonumber\\
      &\ \ge \  \sum_{i=1}^{L}\frac{c_i-1}{\lambda_1^i}\cdot 0 \ = \  0
  \end{align}
  as desired.

  Finally suppose $g=2$. Recall
  \begin{align}
    \label{eq:p_2nk-copy}
    p_{g,n,k}
      &\ = \  \sum_{i=1}^{L}c_{i}p_{g,n-i,k}
        + \sum_{i=1}^{L}c_{i}^*
          \pa{
            \pa{p_{g,n+1-i-g,k-1} - p_{g,n+1-i-g,k}}
            - \pa{p_{g,n-i-g,k-1} - p_{g,n-i-g,k}}
          }.
  \end{align}
  Substituting for $C_\sigma^*$ as before gives
  \begin{align}
    \label{eq:p_gnk-C_sigma-step-1}
    C_\sigma^*
      &\ = \  \sum_{i=1}^{L}c_i\cdot\frac{(0-C_\mu i)^2}{\lambda_1^i}
      \nonumber\\
      &+ \sum_{i=1}^{L}c_{i}^*
          \pa{
            \frac{(1-C_\mu (i+g-1))^2 - (0-C_\mu(i+g-1)^2)}{\lambda_1^{i+g-1}}
             - \frac{(1-C_\mu (i+g))^2 - (0-C_\mu(i+g)^2)}{\lambda_1^{i+g}}
          }
      \nonumber\\
      &\ = \  \sum_{i=1}^{L}c_i\cdot\frac{(C_\mu i)^2}{\lambda_1^i}
      + \sum_{i=1}^{L}c_{i}^*
          \pa{
            \frac{1-2C_\mu (i+g-1)}{\lambda_1^{i+g-1}}
             - \frac{1-2C_\mu (i+g)}{\lambda_1^{i+g}}
          }
          \nonumber\\
      &\ = \  \sum_{i=1}^{L}\frac{c_i}{\lambda_1^{i+g}}\cdot
          \pa{
            \lambda_1^g(C_\mu i)^2
            + \frac{c_{i}^*}{c_i}\pa{
                \lambda_1(1-2C_\mu (i+g-1))
              - (1-2C_\mu (i+g))
            }
          }.
  \end{align}
  Since $\lambda_1>1$, the coefficient of $i$ in
  $\lambda_1(1-2C_\mu (i+g-1)) - (1-2C_\mu (i+g))$
  is negative, so it is minimized when $i=L$.
  Thus if $\lambda_1(1-2C_\mu (L+g-1)) - (1-2C_\mu (L+g))\ge 0$,
  \eqref{eq:p_gnk-C_sigma-step-1} tells us $C_\sigma^* > 0$.
  Thus we may assume $\lambda_1(1-2C_\mu (L+g-1)) - (1-2C_\mu (L+g))< 0$,
  Since $c_i^*/c_i\le 1$ with equality if and only if $i\neq L$, we can simplify \eqref{eq:p_gnk-C_sigma-step-1} to get
  \begin{align}
    C_\sigma^*
      &\ > \ \sum_{i=1}^{L}\frac{c_i}{\lambda_1^{i+g}}\cdot
          \pa{
            \lambda_1^g(C_\mu i)^2
            + \pa{
                \lambda_1(1-2C_\mu (i+g-1))
              - (1-2C_\mu (i+g))
            }
          }.
  \end{align}
  Using standard techniques (such as plugging into Mathematica), one can show that
  $x^y(zw)^2 + (x(1-2z(w+y-1))-(1-2z(w+y))\ge 0$
  for all $x\ge 1, y\ge 2, z\ge 0, w\ge 1$,
  and substituting $x=\lambda_1, y=g, z=C_\mu, w=i$ gives
  \begin{align}
    C_\sigma^*
      &\ > \ \sum_{i=1}^{L}\frac{c_i}{\lambda_1^{i+g}}\cdot
          \pa{
            \lambda_1^g(C_\mu i)^2
            + \pa{
                \lambda_1(1-2C_\mu (i+g-1))
              - (1-2C_\mu (i+g))
            }
          }\nonumber\\
      &\ \ge \  \sum_{i=1}^{L}\frac{c_i}{\lambda_1^{i+g}}\cdot 0 \ = \ 0.
  \end{align}
  as desired.

  For every $g$ and every sequence for which $K_{g,n}$ is nontrivial, we've proven \m{C_\sigma^* > 0}, so we can conclude $C_\sigma > 0$ is all of these cases.



\newcommand{\etalchar}[1]{$^{#1}$}

\bigskip

\end{document}